\documentclass[a4paper,12pt]{amsart}
\usepackage{amssymb}
\usepackage{amsmath}
\usepackage{amscd,amsthm,amssymb}
\usepackage{hyperref}

\scrollmode
\usepackage{latexsym}

\def\.{\cdot}

\def\vs{\vskip .6cm}

\def\la{\langle}
\def\ra{\rangle}

\def\beq{\begin{equation}}
\def\eeq{\end{equation}}
\def\bea{\begin{eqnarray*}}
\def\eea{\end{eqnarray*}}
\def\beaa{\begin{eqnarray}}
\def\eeaa{\end{eqnarray}}
\def\ba{\begin{array}}
\def\ea{\end{array}}

\def \RM{\mathbb{R}}

\def \CM{\mathbb{C}}

\def\Ric{\mathrm{Ric}}
\def\id{\mathrm{id}}
\def\be{\begin{equation}}
\def\ee{\end{equation}}

\def\Sym{\mathrm{Sym}}

\def\SU{\mathrm{SU}}

\def\B{\mathrm{B}}
\def\C{\mathbb{C}}

\def\G{\mathrm{G}}
\def\H{\mathbb{H}}

\def\SO{\mathrm{SO}}

\def\End{\mathrm{End}}

\def\vol{\mathrm{vol}}

\def\Sp{\mathrm{Sp}}
\def\Spin{\mathrm{Spin}}

\def\Sym{\mathrm{Sym}}
\def\scal{\mathrm{scal}}
\def\Id{\mathrm{id}}

\def\Gr{\mathrm{Gr}}

\def\T{T}

\def\pr{\mathrm{pr}}

\def\im{\mathrm{Im}}

\def\T{\mathrm{T}}
\def\E{\mu}
\def\iso{\mathfrak{iso}}

\numberwithin{equation}{section}
\newtheorem{epr}{Proposition}[section]
\newtheorem{ath}[epr]{Theorem}

\newtheorem{ecor}[epr]{Corollary}

\theoremstyle{definition}

\newtheorem{ere}[epr]{Remark}

\title{Eigenvalue estimates and stability of positive quaternion-K\"ahler manifolds}

\author{Yasushi Homma, Uwe Semmelmann}

\address{Yasushi Homma\\
Department of Mathematics, School of Science and Engineering, Waseda University, 
3-4-1 Ohkubo, Shinjuku-ku, Tokyo 169-8555, Japan}
\email{homma\_yasushi@waseda.jp}

\address{Uwe Semmelmann\\
Institut f\"ur Geometrie und Topologie \\
Fachbereich Mathematik\\
Universit{\"a}t Stuttgart\\
Pfaffenwaldring 57 \\
70569 Stuttgart, Germany
}
\email{uwe.semmelmann@mathematik.uni-stuttgart.de}

\date{\today}

\begin{document}

\begin{abstract}
In this article we study the stability problem for  positive quaternion-K\"ahler manifolds. We give a description of 
infinitesimal Einstein deformations and destabilising directions in terms of Laplace eigenfunctions and a special
class of symmetric $2$-tensors. We also give improved eigenvalue estimates for the Hodge-Laplacian on 
$2$-forms. On the parallel subbundle $\Sym^2 E \subset \Lambda^2 T M$ we prove a sharp lower bound
for the first non-zero eigenvalue.

\vs


\end{abstract}

\maketitle


%
\section{Introduction}
%

Einstein metrics on a compact, oriented manifold $M$ are critical points of the Einstein-Hilbert functional 
$\mathcal S : \mathcal M_1 \to \RM$  with $ \mathcal S[g] := \int_M \scal_g \vol_g$, 
where $\mathcal M_1$ denotes the space of Riemannian metrics of volume one and $\vol_g$ the
Riemannian volume element.  Einstein metrics are always saddle points of $\mathcal S$, but they
can be local maxima if the functional is restricted to the set  $\mathcal C$ of unit volume metrics with constant scalar curvature. 
Tangent to $ \mathcal C$ is the space of  {\it tt-tensors}, i.e. traceless and divergence-free symmetric $2$-tensors. In the
tangent space $\T_g\mathcal M_1$, this space is complementary to the space of tangent vectors coming
from the action of the diffeomorphism group or conformal changes of the metric (see \cite{SS2026} for further details).

The second variation  of the Einstein-Hilbert functional $\mathcal S$ on tt-tensors can be expressed in terms of 
the Lichnerowicz-Laplacian $\Delta_L$ on symmetric $2$-tensors as
%
$$
\mathcal S_g^{''} (h, h)  = -\frac12 \la \Delta_L h - 2\E h, h\ra_{L^2} \ .
$$
Here $g$ is an Einstein metric of Einstein constant $\E = \frac{\scal_g}{\dim M}$.
The Lichnerowicz-Laplacian  $\Delta_L$ is an elliptic operator of Laplace type. In particular, on compact manifolds, it has finite dimensional eigenspaces
and only finitely many negative eigenvalues
(see \cite{SS2026} or Section \ref{wbf}).

An Einstein metric $g$ is called  {\it strictly  (linearly) stable} if  ${\mathcal S}_g^{''} < 0$  on the space of
tt-tensors, or equivalently if $\Delta_L > 2\E$, and semi-stable if  ${\mathcal S}_g^{''} \le 0$. It is called {\it linearly unstable} 
if there exists a tt-tensor $h$  such that $\mathcal S_g^{''} (h, h) >0$, or equivalently $\Delta_Lh = 2 \lambda h$ for
some eigenvalue $\lambda < \E$. Such tensors $h$ are called  {\it destabilising directions}. 
Unstable Einstein metrics are particularly interesting since they turn out to be also unstable with respect to Perelman's $\nu$-entropy as well as dynamically unstable
with respect to the Ricci flow.
A strictly (linearly) stable 
Einstein metric is in particular $\mathcal S$-{\it stable}, in the sense that $g$ is a local maximum of $\mathcal S$ restricted to $\mathcal C$. Moreover,  {\it infinitesimal Einstein deformations} (IED) are defined as tt-tensors which are zero-directions of   $\mathcal S_g''$, i.e.
tt-tensors $h$ with $\Delta_Lh = 2\E h$. Any curve of Einstein metrics has such an IED as tangent vector (up to  the action of the diffeomorphism group and scaling). Conversely, not every
IED is such a tangent vector.

Let $(M,g)$ be an Einstein manifold with Einstein constant $\E  > 0$ other
than the standard sphere. Then $g$ is $\nu$-stable, i.e., linearly stable with respect to the
$\nu$-entropy,
if and only if
$\Delta \ge 2 \E$ for the Laplace operator 
$\Delta$ on functions 
and
$\Delta_L \ge 2 \E$ for the Lichnerowicz-Laplacian
$\Delta_L$ on tt-tensors (see  \cite{CH15}, Cor. 1.3).

Among Einstein metrics, the stable ones form a distinguished class. Determining the stability of Einstein metrics 
under curvature and holonomy conditions is therefore an important problem. For irreducible symmetric spaces this was
done by Koiso in \cite{Koiso} (see \cite{SW2022} for the clarification of a few open cases). Note that under the assumption of  positive scalar curvature Riemannian
products are unstable. In this case, an explicit destabilising direction can be constructed using the projections to the factors. 
Manifolds with Ricci-flat holonomy $\G_2, \Spin_7, \SU(n)$ and $\Sp(n)$ are all semi-stable by 
\cite{DWW07}. 
However, it is unknown whether general Ricci-flat manifolds have to be semi-stable. It is a long-standing open problem whether
Ricci-flat manifolds with generic holonomy exist. K\"ahler-Einstein manifolds of negative scalar curvature are all semi-stable
by \cite{DWW05}. K\"ahler-Einstein manifolds with positive scalar curvature are unstable if the second Betti number is
larger than one. 

In this article we  study the stability  of manifolds with the  remaining  non-generic holonomy group 
in the Berger list of holonomies of non-symmetric, irreducible Riemannian manifolds: the so-called quaternion-K\"ahler manifolds.
These are Riemannian manifolds $(M^{4n}, g)$ such that 
the holonomy group is a subgroup of $\Sp(1)\cdot \Sp(n) \subset \SO(4n)$. For $n\ge 2$ they are automatically
Einstein and they are irreducible, if the scalar curvature is different from zero.
Quaternion-K\"ahler manifolds of negative scalar curvature 
are stable by a simple argument (see \cite{KS25}, Thm. 1.11). 

From now on we  assume the scalar curvature 
to be positive, i.e. we will only consider the case of so-called {\it positive quaternion-K\"ahler manifolds}, where
we also assume completeness. 
The only known examples in this situation are the 
{\it Wolf spaces}. These are certain symmetric spaces of compact type, one for each simple Lie algebra. Examples
are the quaternion projective space $\H P^n$, the complex $2$-plane Grassmannian 
$\mathrm{Gr}_2(\CM^{n+2})$ or the exceptional symmetric space $\mathrm G_2/\SO(4)$.
By the work of Koiso (see \cite{Koiso} and \cite{SW2022})  all Wolf spaces different from $\Gr_2 (\C^{n+2})$ 
are strictly (linearly) stable, i.e. there are no destabilising directions. The complex Grassmannian $\Gr_2 (\C^{n+2})$ is
semistable and admits infinitesimal Einstein deformations.

Finally we recall the  LeBrun-Salamon conjecture, stating that all positive quaternion-K\"ahler manifolds are symmetric,
i.e. Wolf spaces. So far, the conjecture is confirmed in quaternionic dimensions $n = 2, 3$ and $4$  
(see \cite{PS}, \cite{BW}).

\medskip

The main result of our article gives a description of infinitesimal Einstein deformations and 
destabilising directions on a positive quaternion-K\"ahler manifold in terms of eigenfunctions 
of the Laplacian and a special kind of symmetric $2$-tensors in $\Sym^2H \otimes  \Sym^2E \subset \Sym^2 TM$,
which will be defined in  Section \ref{QK}. Then we have (see also Section \ref{final1}):

\noindent
{\bf Theorem A.}
{\it
Let $(M^{4n}, g)$ be a positive quaternion-K\"ahler manifold. Then:
\begin{enumerate}
\item
The space of infinitesimal Einstein deformations of $g$
is isomorphic to the eigenspace of the Laplace operator $\Delta$ on functions
for the eigenvalue $\lambda_2 = 2\E   = \frac{\scal_g}{2n\,} $.
\medskip
\item
The space of destabilising directions of the Einstein metric $g$ is isomorphic to
$$
\ker(\Delta_{\Sym^2H \otimes \Sym^2E}-\lambda_1)  \oplus   \,  \bigoplus_{\lambda_1 < \lambda < \lambda_2}  \ker (\Delta - \lambda) \ ,
$$
where $\lambda_1 =  \frac{n+1}{n+2} 2\E$  and
$\Delta_{\Sym^2H \otimes \Sym^2E}$ is the restriction of the Lichnerowicz-Laplace operator  $\Delta_L$ to the bundle 
$\mathrm{Sym}^2H  \otimes \Sym^2E \subset \Sym^2 TM$.
\end{enumerate}
}

\medskip

We recall that on positive quaternion-K\"ahler manifolds  $\lambda_1 =  \frac{n+1}{n+2} 2\E$ is the lower bound 
for the Laplace operator $\Delta$ on non-constant functions.
It is attained if and only if $M$ is up to scaling isometric to  $\H P^n$ (see \cite{AM}, \cite{LeB}). For all Wolf spaces 
different from  $\H P^n$ one has the lower bound $\Delta \ge \lambda_2=2\E$, with equality for the complex Grassmannian (see \cite{Mil}). Moreover,  for positive quaternion-K\"ahler manifolds the index $i^{1,n+1}$ of the Dirac operator twisted with  $\Sym^{n+1} H E$ is given as $-\dim \ker(\Delta_{\Sym^2H \Sym^2E}-\lambda_1) + \dim \ker (\Delta - \lambda_1)$.
It is conjectured that for all positive quaternion-K\"ahler manifolds different from $\H P^n$ this index, and thus
$\dim \ker(\Delta_{\Sym^2H \Sym^2E}-\lambda_1) $ vanishes (see \cite{SS} and Remark \ref{wolf} below).

Note that there is a very similar result for the squashed Einstein metric on $7$-dimensional  $3$-Sasaki manifolds.
Indeed, Thm. 1.1 and  Thm. 1.3 in \cite{NS24} give a description of infinitesimal Einstein deformations and
destabilising directions in terms of Laplace eigenfunctions and special symmetric $2$-tensors. However, these
metrics are always unstable.

As a consequence of our theorem and results of  \cite{herrera} and \cite{Mil}  on the eigenspaces appearing in the theorem 
we can reprove Koiso's stability statement for the Wolf spaces. Conversely, using the Koiso result and our
theorem we can confirm calculations in \cite{herrera}, in particular the vanishing of the index $i^{1,n+1}$
on the Wolf spaces (see Section \ref{final} and Appendix A for more details).

For the proof of our main theorem we investigate eigenspaces of the Lichnerowicz-Laplacian $\Delta_L$ on
the two parallel subbundles $\Lambda^2_0E$ and $\Sym^2H\otimes\Sym^2E$ of $\Sym^2 TM$. Our main tools are
Weitzenb\"ock formulas for quaternion-K\"ahler manifolds developed in \cite{H1}. 

Using the same methods
we find improved lower bounds for the eigenvalues of the Hodge-Laplacian on $2$-forms.
Particularly interesting is the parallel subbundle $\Sym^2 E$ of $\Lambda^2 \T M$. 
It is known that harmonic $2$-forms have to be sections of this bundle, i.e. the lower bound is zero (see \cite{SW2002}). 
However, it can be shown  that $\mathrm{Gr}_2 (\C^{n+2})$ is the only positive quaternion-K\"ahler manifold admitting
harmonic $2$-forms (see \cite{LS}).
In Section \ref{sym2} we obtain a lower bound for the first non-vanishing eigenvalue of the Hodge-Laplacian:

\noindent
{\bf Theorem B.}
{\it Let $(M^{4n}, g)$ be a positive quaternion-K\"ahler manifold. Then the first eigenvalue  $\lambda \neq 0$ of
the Hodge-Laplacian on sections of $ \Lambda^{1,1}_0E \cong \Sym^2 E \subset \Lambda^2 \T M$ satisfies $\lambda \ge \frac{\scal}{2(n+2)}$.
}

\begin{ere}
In fact it is possible to show that
the symmetric space  $\mathrm G_2/\SO(4)$ is the only Wolf space where equality in the estimate is attained.
\end{ere}

%
%
%
%

%
\section{Preliminaries}
%

\subsection{Quaternion-K\"ahler manifolds}\label{QK}

A quaternion-K\"ahler manifold is a Riemannian manifold $(M^{4n}, g)$ such that
the holonomy group of the Levi-Civita connection is contained in $\Sp(1) \cdot \Sp(n)$.
Equivalently there are locally defined almost complex structures $I, J, K$ compatible
with the metric, satisfying the quaternion relations $I^2 = J^2 = K^2 = - \Id, I J = K$, and  spanning a 
globally defined, parallel, rank $3$
subbundle $Q \subset \End \, \T M$.

In this article we will use the $H E$ formalism of Salamon (see \cite{SS}). Here $H = \H^1$ 
denotes the standard representation of $\Sp(1)$ and $E = \H^{n}$ the standard representation of
$\Sp(n)$. The holonomy reduction of a quaternion-K\"ahler manifold defines a $\Sp(1) \cdot \Sp(n)$ 
orthonormal fibre bundle and each $\Sp(1) \cdot \Sp(n)$-representation defines an associated 
vector bundle. In the following we will use the same notation for the bundle and the
defining representation. 
Geometric vector bundles can be expressed in terms of $H$ and $E$, e.g. the complexified tangent bundle 
is given as  $TM^\C \cong H \otimes E$ and we have $Q \cong \Sym^2 H$ for the
quaternionic structure bundle. 
Let $\Lambda^{a,b}_0E$ be the Cartan summand in the tensor product $\Lambda^a_0E \otimes \Lambda^b_0E $.
Here $\Lambda^a_0E \subset \Lambda^a E$ denotes the primitive part, i.e. the kernel of the contraction with
the symplectic form $\sigma_E$ of $E$. The representations  $\Lambda^a_0E$ are the fundamental (irreducible) 
representations of $\Sp(n)$. Note that $\Lambda^{1,1}_0E \cong \Sym^2 E$. It is easy to see that  the space of forms 
decomposes into a sum of irreducible representations all of the form $\Sym^kH \otimes \Lambda^{a,b}_0E$.
Below, in \eqref{Lambda2} resp. \eqref{Sym2}, we give the explicit decompositions of $\Lambda^2 \T M$ and $\Sym^2 \T M$
into irreducible summands.

The Riemannian curvature tensor $R$ of a quaternion-K\"ahler manifold can be written as
$R = R_{\H P^n } + R^{hyper}$, where $R_{\H P^n } $ is the curvature tensor of the quaternion
projective space $\H P^n$ and $R^{hyper}$ is a curvature tensor of hyper-K\"ahler type.  In particular its
Ricci curvature vanishes. This shows that quaternion-K\"ahler manifolds are Einstein. We note that the tensor $R^{hyper}$  can be identified with a section of $\Sym^4 E$ (see \cite{SW2002} for details).

An important tool for the investigation of positive quaternion-K\"ahler manifold  $M$ is the twistor space $ZM$. It is defined as
the sphere bundle of the quaternionic structure bundle $Q$, i.e.
$ZM : = \{aI +bJ +cK \,|\, a^2 +b^2+c^2 =1\}$, where $I,J,K$ is a local quaternionic frame spanning $Q$.  Then $ZM$ admits
an integrable complex structure and  a K\"ahler-Einstein metric of positive scalar curvature, for which the canonical 
projection $\pi : ZM \rightarrow M$ is a Riemannian submersion, with fibres isometric to $\CM P^1$. The twistor space
of $\H P^n$ is $\C P^{2n+1}$.

We conclude with the following useful remark. A non-symmetric positive quaternion-K\"ahler manifold
$(M^{4n},g)$ has full holonomy $\mathrm{Sp}(1)\cdot \mathrm{Sp}(n)$. Indeed, such a manifold
has positive scalar curvature, is irreducible, and its holonomy group is contained in
$\mathrm{Sp}(1)\cdot \mathrm{Sp}(n)$. If the holonomy were a proper subgroup, then by the
Berger holonomy theorem and dimensional reasons it could not occur as the holonomy of a
non-symmetric irreducible manifold.
As a consequence, if $M$ has full holonomy $\mathrm{Sp}(1)\cdot \mathrm{Sp}(n)$, then there
are no non-trivial parallel $2$-forms and every parallel symmetric $2$-tensor is a constant
multiple of the metric $g$. This follows since there is no trivial summand in the
representations defining $\Lambda^2 \ TM$ and $\Sym^2_0 \T M$ (trace free part).


\subsection{Generalized gradients and Weitzenb\"ock formulas}  \label{wbf}

In this section we recall the definition of generalized gradients and the theory of  Weitzenb\"ock formulas on quaternion-K\"ahler manifolds as developed in \cite{H1}. Generalized gradients are $1$st order differential
operators  on sections of a bundle $VM$. They are defined as orthogonal projections of the covariant derivative 
induced by the Levi-Civita connection to irreducible summands of $TM \otimes VM$. 

Irreducible $\Sp(n)$-representations $V_\rho$ are parameterised by their highest weight $\rho$  
satisfying the dominant integral condition,
\beq\label{dom}
\rho=(\rho_1,\dots,\rho_n)\in \mathbb{Z}^n  \;\textrm{ and } \; \rho_1\ge \rho_2\ge\cdots \ge \rho_n\ge 0 \ .
\eeq
Any $\Sp(1) \cdot \Sp(n)$-representation is given as $\Sym^kH \otimes V_\rho$, with $k\ge 0$ and $\rho$ with
\eqref{dom}. We have the following decomposition 
into irreducible summands
\beq\label{deco}
(\Sym^kH \otimes V_\rho) \otimes (H \otimes E) 
\quad \cong 
\bigoplus_{\substack{\scriptstyle N=\pm 1, \\ \scriptstyle \nu=\pm 1,\dots,\pm n }}\Sym^{k + N}H \otimes V_{\rho + e_{\nu}} \ ,
\eeq

where $\{e_\nu\}$ is the standard basis of $\RM^n$ and we write $e_{-\nu} = - e_\nu$. On the right side 
of \eqref{deco} we leave out summands where the dominant integral condition is not satisfied. We denote the orthogonal projection onto the summand 
$\Sym^{k + N}H \otimes V_{\rho + e_{\nu}} $ with $\Pi_{N, \nu}$. Then we can define the generalized gradient
$$
D_{N, \nu} 
:= \Pi_{N, \nu} \circ \nabla : \Gamma (\Sym^kH \otimes V_\rho) \rightarrow  \Gamma(\Sym^{k+N} H \otimes V_{\rho + e_\nu}) \ ,
$$
where $\nabla$ is induced by the Levi-Civita connection. The formal adjoint of $D_{N, \nu}$ is denoted by
\[
(D_{N, \nu})^{\ast}:
\Gamma(\Sym^{k+N}H \otimes V_{\rho + e_\nu})   \to    \Gamma(\Sym^{k} H \otimes V_{\rho }) \ .
\] and the composition of $D_{N, \nu}$ and $(D_{N, \nu})^{\ast}$ (i.e. a 2nd order operator) by
\[
B_{N, \nu}=(D_{N, \nu})^{\ast} \, D_{N, \nu} :  \Gamma(\Sym^{k} H \otimes V_{\rho }) \to  \Gamma(\Sym^{k} H \otimes V_{\rho }) \ .
\]

Consider the generalized gradient $D_{-N, -\nu} : \Gamma(\Sym^{k+N}H \otimes V_{\rho + e_\nu})   \to    \Gamma(\Sym^{k} H \otimes V_{\rho }) $. Its formal adjoint maps between the same spaces as $D_{N, \nu}$. The next proposition then gives
the so-called {\it relative dimension formula}, a very helpful relation between the $2$nd order operators
$B_{N, \nu}  = (D_{N, \nu})^* D_{N, \nu}$ and $D_{-N, -\nu} \, D_{-N, -\nu}^*$. It was  stated and proved for general 
Riemannian manifolds (see \cite{HT}). However, with essentially the same proof,  we have the corresponding statement
for quaternion-K\"ahler manifolds.
\begin{epr}
On sections of the bundle $\Sym^kH \otimes V_\rho$ the following formula holds
\[
(D_{N,\nu})^{\ast}D_{N,\nu}=\frac{\dim V_{k+N,\rho+e_{\nu}} }{\dim V_{k,\rho}}D_{-N,-\nu}(D_{-N,-\nu})^{\ast}.  
\]
\end{epr}

On sections of a vector bundle $VM$ we have the second order operator $\nabla^*\nabla$ and also the so-called
{\it standard Laplace operator} $\Delta_V := \nabla^*\nabla + q(R)$. Here $q(R)$ is a symmetric
endomorphism of $VM$ defined by 
$
q(R) : =  \sum_{i<j} d\pi(\pr_{\mathfrak{sp(1)sp(n)} }(X_i \wedge X_j)) \circ d\pi (R(X_i \wedge X_j))
$,
where $\{X_i\}$ is an orthonormal basis of $TM$ and $d\pi$ denotes the differential of the $\Sp(1)\Sp(n)$-representation  $\pi$ defining the bundle $VM$ (see \cite{SW2019} for details). 

The operator $\Delta_V$ commutes with parallel bundle maps. Hence, it only depends on the defining representation. If $\Phi : VM \to WM$ is such a parallel bundle map 
(induced by an equivariant map between the representations $V$ and $W$), then $\Phi \circ \Delta_V = \Delta_W \circ \Phi$. 
In particular, on any parallel subbundle $VM$ of the form bundle the standard Laplacian $\Delta_V$  coincides 
with the restriction to $\Gamma(VM)$ of the Hodge-Laplacian $\Delta = d d^* + d^* d $. Similarly,  on any parallel subbundle of the bundle of symmetric tensors the standard Laplacian coincides with the restriction of the Lichnerowicz-Laplacian $\Delta_L$.
Note that the  Lichnerowicz-Laplacian is the standard Laplacian for the bundle of symmetric tensors. Particularly interesting
for us will be the case of symmetric $2$-tensors, where the curvature endomorphism $q(R)$ can also be written as
$
q(R) = 2\mathring R +  \Ric
$,
with $(\mathring R h)(X,Y) = \sum_i h(R_{X, X_i}Y, X_i)$, for a symmetric $2$-tensor $h$. Here $\Ric$ acts as derivation, i.e.
$\Ric = 2\E \Id$ for an Einstein manifold with Einstein constant $\E$.

A {\it Weitzenb\"ock formula} is any linear combination of operators $B_{N, \nu}$, which gives a zero order operator. Weitzenb\"ock
formulas form a vector space. Its dimension is half the number of summands in the decomposition \eqref{deco}. Roughly 
half of the Weitzenb\"ock formulas involve the curvature term $R^{hyper}$ and cannot be used directly. For any given
$\Sp(1)\Sp(n)$-representation $V$ all Weitzenb\"ock formulas on sections of the associated bundle $VM$ are known
(see \cite{H1}). 

As a {\it universal Weizenb\"ock formula}, defined on all bundles $VM = \Sym^kH \otimes V_\rho$, we have
\beq\label{universal}
q(R) = \sum_{N, \nu} w_{N, \nu}\, B_{N, \nu}  \qquad \mbox{with} \quad w_{N, \nu } =  \frac{W_N}{2n } + \frac{w_\nu}{2} \ ,
\eeq
where $W_{+1} = -k$ and $W_{-1} = k+2$ and $w_\nu$ is defined as follows. Let $\rho = (\rho_1, \ldots, \rho_n)$ be 
the highest weight. Then the conformal weight $w_\nu$  is defined by
\begin{equation}
w_{\nu}:=-\rho_\nu  + \nu - 1   \textrm{ for $\nu=1,\dots,n$ \quad and \quad }  w_{-\nu}:=- w_\nu + 2n   \;\; \textrm{for $\nu=1,\dots,n$}  \nonumber \ .
\end{equation}
The curvature endomorphism $q(R)$ depends linearly on the two components of the curvature $R_{\H P^n}$ and $R^{hyper}$, 
where  $R^{hyper}$ only acts on the $\Sp(n)$-part of the representation. For many representations $q(R)$ is
explicitly known (see \cite{SW2002}, note that our $q(R)$ is twice the $q(R)$ in \cite{SW2002}). Recall that $R^{hyper}$ acts trivially on all representations $\Lambda^a_0E$ (see \cite{KSW}, Lemma 2.5).

Since $\nabla$ can be written as the sum of all  generalized gradients we have 
$
\nabla^*\nabla = \sum_{N, \nu}  \, B_{N, \nu} 
$.
Hence, we also have $\Delta = \sum_{N, \nu} (1 + w_{N, \nu}) \, B_{N, \nu}   $, which with a slight abuse of notation 
is also called  universal  Weitzenb\"ock formula.

The most important application of Weitzenb\"ock formulas are lower bounds for the spectrum of the standard
Laplace operator $\Delta$ (see \cite{MFO} and \cite{H1}, with an improvement for $k=0$).

\begin{epr}
The eigenvalues of the standard Laplacian  $\Delta$ on  sections of $\Sym^k H \otimes \Lambda^{a,b}_0  E$ 
for $0\le k\le 2n-a-b$ and $0\le b\le a\le n$ have the following lower bounds
\[
\begin{cases}
\displaystyle{(a-b+k)(2n-a-b+k+2) }\frac{\scal}{8n(n+2)} &  \quad \textrm{for $k\neq 0$}\\[2.5ex]
\displaystyle{(a-b)(2n-a-b+4)} \frac{\scal}{8n(n+2)}  & \quad  \textrm{for $k=0$}
\end{cases}
\]
\end{epr}

We will call the lower bound given in the above proposition the {\it minimal eigenvalue} of the bundle 
$ \Sym^k H \otimes \Lambda^{a,b}_0  E$. In the proof of the proposition (see \cite{H1}, Thm. 8.5) 
a combination of Weitzenb\"ock formulas expresses the standard Laplacian as the minimal 
eigenvalue plus a linear combination of operators $B_{N, \nu}$ with positive coefficients. In particular,
on $\Delta$-eigensections   for the minimal eigenvalue the  generalized gradients appearing in this
combination
have to vanish. However, only in a very few cases the given eigenvalue estimate is known to be sharp.

For later application we still need one other important result on generalized gradients,
the so-called {\it commutator formula}. For any Riemannian manifold it is known
that the Hodge-Laplace operator commutes with the differential $d$ and the codifferential
$d^*$. Similarly, the Lichnerowicz-Laplacian on symmetric tensors commutes with the
divergence $\delta$ and its formal adjoint $\delta^*$. On a quaternion-K\"ahler manifolds
there are many more generalized gradients commuting with the standard Laplacian.
The following proposition is a special case of a general commutator formula given in \cite{SW2019}, Thm. 3.1.

\noindent
\begin{epr}\label{commutator}
Let $(M, g)$ be a quaternion-K\"ahler manifold. Then any generalised gradient 
$
D: \Gamma (\mathrm{Sym}^k H \otimes \Lambda_{0}^{a,b} E  ) \rightarrow \Gamma( VM )
$
with a  geometric vector bundle \( V M \), corresponding
to a representation of the form
$\,
V = \mathrm{Sym}^{k \pm 1} H \otimes \Lambda_{0}^{a \pm 1, b} E
\,$
or
$\,
V = \mathrm{Sym}^{k \pm 1} H \otimes \Lambda_{0}^{a, b\pm1} E
\,$
commutes with the corresponding standard Laplacian $\Delta$.
\end{epr}

Our standard application of  the commutator formula will be as follows. For any
 generalized gradient  $D: \Gamma(VM) \to \Gamma (WM)$ commuting with
the standard Laplacian, we have $D\varphi = 0$ on all $\Delta_V$-eigensections  $\varphi$
for eigenvalues less than the minimal eigenvalue of $WM$.
In this situation we also use the following notation. We write $(VM)_\lambda = \Gamma(VM) \cap \ker (\Delta_V - \lambda)$.
Then we have $D : (VM)_\lambda \to (WM)_\lambda$.
As a special case of this notation we will write $C^\infty(M)_\lambda $ for the eigenspace of the Laplace operator on functions
for the eigenvalue $\lambda$.

\medskip

%
%
%
%
%
%


%
\section{Eigenvalue estimates and Weitzenb\"ock formulas on   \texorpdfstring{$\T M$}{\T M}}
%

We start with  a compact Riemannian  manifold $(M^m, g)$ and assume $g$ to be an Einstein metric,  i.e. $\Ric = \E g$, with positive Einstein constant
$\E = \frac{\scal}{m}$.   Let  $X$ be a vector field  on $M$ with 
$\Delta X = \lambda X$, i.e. $X \in   (\T M)_\lambda$. For any $\lambda \neq 0$ we have 
$ C^\infty(M)_\lambda \cong d \, C^\infty(M)_\lambda \subset (\T M)_\lambda$, where the isomorphism 
is given by the differential $d$. The following facts are well-known.
\begin{enumerate}
\item
If the vector field $X$ is coclosed, then
 $\lambda \ge \lambda_2 := 2\E$. Equality holds if and only if $X$ is a Killing vector field.
\medskip
\item
If $\lambda < \lambda_2$ then
$ 
(\T M)_\lambda = d \, C^\infty(M)_\lambda \cong  C^\infty(M)_\lambda \ .
$
\medskip
\item
If $\lambda = \lambda_2 $, then $(\T M)_\lambda = \iso(M,g) \oplus d \, C^\infty(M)_\lambda$, with 
$\iso(M,g)$ denoting the space of Killing vector fields  of  $(M, g)$.
\medskip
\item
(Lichnerowicz-Obata) $\lambda \ge \frac{m}{m-1} \mu $, with equality exactly for the standard sphere
\end{enumerate}

From now on let $(M^{4n}, g)$ be a positive quaternion-K\"ahler manifold. For later use we recall the improved Lichnerowicz-Obata estimate for Laplace eigenvalues on non-constant  functions and equivalently  on vector fields.
\begin{epr}{\cite{AM}, \cite{LeB}}\label{LeB}
Let  $(M^{4n}, g)$ be a positive quaternion-K\"ahler manifold. Then any eigenvalue $\lambda$ of the Laplace operator on vector fields satisfies
$\lambda \ge \lambda_1 := \frac{n+1}{n+2} 2\E$,  with equality if and only if $M$ is up to scaling isometric to $ \H P^n$.  The same statement is true if
$\lambda$ is the first non-zero eigenvalue of the Laplace operator on functions.
\end{epr}
\proof
We give a short argument for the estimate on functions and vector fields and refer to \cite{AM} or  \cite{LeB} for the discussion of the equality case. 

First we prove the estimate on functions. For  $f \in  C^\infty(M)_\lambda$ the pull-back of $f$ to the twistor space $ZM$  is again a 
$\Delta$-eigenfunction
for the same eigenvalue $\lambda$. On K\"ahler-Einstein manifold of positive scalar curvature  one has a well-known eigenvalue estimate, which implies
 $\lambda \ge 2\E_Z$, where $\mu_Z$ is the Einstein constant of the K\"ahler-Einstein metric on $ZM$. The estimate now follows by  checking that $\E_Z$ is exactly the lower bound $\lambda_1$
given in the statement.
 
 Next, consider $X \in (\T M)_\lambda$. Then either $X$ is co-closed and then $\lambda \ge 2\E > \lambda_1$, or $X$ is not co-closed. But then
 $f: = d^*X$ is a non-vanishing,  non-constant eigenfunction for the eigenvalue $\lambda$ and  $\lambda \ge \lambda_1$ by the estimate on functions.
 \qed
%

\medskip

\subsection{Weitzenb\"ock formulas on   \texorpdfstring{$HE$}{HE}}
For latter applications we have to study Weitzenb\"ock formulas on $\T M$ in more detail. To simplify the notation, we will omit the tensor product symbol between the Sp(1) and Sp(n) representations in what follows. On  the vector bundle $\T M^\CM = H E$ one has six generalised gradients,
which we will always consider restricted to  $(HE)_\lambda$, with $\lambda \le \lambda_2$. Then the definition of the generalised gradients and the commutator 
rule or Proposition  \ref{commutator}  imply:
$$
 \begin{array}{llcllc}
     D_{-1,1} :    & (HE)_\lambda  \longrightarrow & (\Sym^2 E)_\lambda             & \qquad  D_{1,1} :    & (HE)_\lambda  \longrightarrow  &(\Sym^2H\Sym^2 E)_\lambda  \\
    D_{-1,2}  :   & (HE)_\lambda  \longrightarrow &(\Lambda^2_0 E)_\lambda    &  \qquad D_{1,2} :    & (HE)_\lambda  \longrightarrow & (\Sym^2 H \Lambda^2_0E)_\lambda    \\
      D_{-1,-1} : & (HE)_\lambda  \longrightarrow  &C^\infty(M)_\lambda            &  \qquad D_{1,-1} :    & (HE)_\lambda  \longrightarrow  & (\Sym^2 H)_\lambda  
\end{array} 
$$

Since the minimal eigenvalue on $\Sym^2 H\Lambda^2_0E$ is larger than $\lambda_2$ we can assume that $D_{1,2} = 0$ on $ (HE)_\lambda $. The minimal 
eigenvalue on $\Sym^2 H$ is $\lambda_2$. Hence,  $D_{1,-1} = 0$ on $ (HE)_\lambda $ if $\lambda < \lambda_2$.

Assuming that $D_{1,2}$ vanish we obtain the following
three Weitzenb\"ock formulas (see \cite{H1}):
\medskip
{\small
$$
 \begin{array}{llllllc}
 (2n+1)B_{1,-1 } &- B_{1,1} & - B_{-1,1}  &+ B_{-1,2}  &+ (2n+1)B_{-1,-1}  & = & \tfrac{2n+1}{n+2} \, \E \\[1ex]
-B_{1,-1 }& - B_{1,1} & + 3 B_{-1,1} &+ 3B_{-1,2} &+ 3 B_{-1,-1}  & = & \frac{3n}{n+2} \, \E \\[1ex]
-n(2n+1)B_{1,-1 }&- (n+2) B_{1,1} & + 3(n+2)B_{-1,1} &  -3n B_{-1,2} & + 3n (2n+1)B_{-1,-1}  & = & \frac{3(2n+1)}{n+2} \, \E
\end{array} 
$$
}
Assuming in addition that $D_{1,-1}$ vanishes,  these three equations show in a first step the vanishing of $B_{-1,1} $. In a second step,
by subtracting the first two equations, we obtain the Weitzenb\"ock formula
\beq\label{eq1}
B_{-1,2} \; - \;  (n-1)B_{-1,-1}  \; = \; \frac{n-1}{2(n+2)} \, \E \ .
\eeq
A first consequence of this equation is that $D_{-1,2} X  $ cannot vanish identically if $X \neq 0$. Since otherwise
integration  would give  $-(n-1)\|D_{-1,-1} X\|^2 =  \frac{n-1}{2(n+2)} \E \|X\|^2$, which contradicts $X \neq 0$.
At  the same time we see that $D_{-1,2}   $ is injective on $(HE)_\lambda$ for $\lambda_1 \le \lambda < \lambda_2$.

In addition to the three Weitzenb\"ock formulas above we also have the universal formula 
 $\nabla^*\nabla = \sum_{N, \nu } B_{N, \nu}$. Hence, for a vector field $X \in (\T M)_\lambda$, with $\lambda < \lambda_2 $ we obtain
\beq\label{eq2}
\lambda X \;=\; \Delta X \;=\; \nabla^*\nabla X \, + \, q(R)X \;=\; B_{1,1}X \,  +\,  B_{-1,2}X  \,+ \, B_{-1,-1}X  \,+\, \frac{\scal}{4n}X \ .
\eeq
Recall that for $HE$ we have $q(R)= \Ric = \frac{\scal}{4n}$. We also use that under the assumption $\lambda < \lambda_2 $  the other three generalised 
gradients vanish on the vector field $X$. 
Finally, combining the second equation in the system above with
\eqref{eq1} and \eqref{eq2} leads to
\beq\label{eq3}
B_{-1,2} \,X \;=\; \frac{n-1}{4n} \, \left( \lambda \,+\, \frac{\scal}{2(n+2)} \right) X\ .
\eeq
Similarly we also obtain an equation for $B_{1,1}$. Here we find
\beq\label{eq4}
B_{1,1} \,X \;=\; \frac{3}{4} \, \left( \lambda \,-\, \frac{n+1}{2n(n+2)} \scal \right) X\ .
\eeq

\medskip

Let us now consider the special case $X = df \in (\T M)_{\lambda_2}$. Since $X$ is closed
the gradients $D_{-1,1}, D_{1,-1} $ and $D_{1,2} $ vanish on $X$ by Corollary \ref{closed}.
We again obtain \eqref{eq3} and \eqref{eq4} for $X$.

\subsection{Generalised gradients on   \texorpdfstring{$HE$}{HE} compared with differential and divergence}

%


\begin{epr}\label{differential}
Let  $(M^{4n}, g)$ be a quaternion-K\"ahler manifold, then on $ HE$ it holds that
$$
d \, d^* \, = \, 4n \, B_{-1,-1},  \;
d^* \,  \pr_{\Sym^2 H }  \, d \, =\,  2 \, B_{1,-1},
\;
d^*  \, \pr_{\Sym^2 E} \, d \, =\,  2 \, B_{-1,1}, 
\;
d^*  \, \pr_{\Sym^2H \Lambda^2_0 E} \, d \, =\,  2 \, B_{1,2}  
\ .
$$
\end{epr}
\proof
On a general  Riemannian manifold $(M^m, g)$ there are generalised gradients
$$
Q_1 : \Gamma(\Lambda^p \T M) \rightarrow  \Gamma (\Lambda^{p+1} \T M)
\qquad \mbox{and} \qquad
Q_2 : \Gamma(\Lambda^p \T M) \rightarrow  \Gamma (\Lambda^{p-1} \T M)
$$
where $\Lambda^{p+1} \T M$ and $ \Lambda^{p-1} \T M$ are considered as
subspaces of $\T M \otimes \Lambda^{p} \T M$ and the differential  operators $Q_1$ and $Q_2$
are defined by composing the covariant derivative with the orthogonal projection onto
these two subspaces. The following relation is  well known (see \cite{GM} or \cite{S}):
\beq\label{compare}
  d d^*  = (m-p+1) \, Q_2^* Q_2    \qquad \mbox{and} \qquad   d^* d = (p+1) \, Q_1^* Q_1   \ .
\eeq

Let now $(M^{4n}, g)$ be a quaternion-K\"ahler manifold and consider the case $p=1$. 
Then the operators $Q_2$ and $D_{-1,-1}$ coincide and the first formula in the proposition 
follows from the first equation in \eqref{compare} for $m=4n$ and $p=1$.

Next we consider the operator $Q_1$. We have $Q_1 = c \, d$ for some complex number $c$ and the second
equation in \eqref{compare} shows $|c|^2 = \frac12$. The generalised gradients $D_{1,-1}, D_{-1,1}$ and $D_{1,2}$ 
are defined as the orthogonal projection of the covariant derivative on $HE$ to the three summands
$ \Sym^2H,  \Sym^2E$ and $\Sym^2H\Lambda^2_0E$ of the irreducible decomposition of $\Lambda^2 \T M$.

Since all generalised gradients are defined with orthogonal projections it follows that 
$D_{1,-1} = \tilde c \, \pr_{\Sym^2H} \, Q_1$ for some $\tilde c \in \C$ with $|\tilde c|^2 = 1$. Hence
$D_{1,-1} =  \tilde c \, c \, \pr_{\Sym^2H} \,d $ and thus
$$
B_{1,-1} = D^*_{1,-1} D_{1,-1}= |c|^2 (\pr_{\Sym^2H} \,d)^* (\pr_{\Sym^2H} \,d) = \frac12 d^* \pr_{\Sym^2H}^* \pr_{\Sym^2H} d 
= \frac12 d^*  \pr_{\Sym^2H} d  \ .
$$
This proves the second equation in the proposition. The remaining two equations follow by the same argument.
\qed

\medskip

As a direct consequence of Proposition \ref{differential} we have the following

\begin{ecor}\label{closed}
Let  $(M^{4n}, g)$ be a positive quaternion-K\"ahler manifold, then 
\begin{enumerate}
\item 
$D_{-1,1} = D_{1,-1} = D_{1,2}  = 0$ on closed vector fields.
\medskip
\item
 $D_{-1,-1}  = 0$ on coclosed vector fields.
\medskip
\item
$D_{-1,2} = D_{1,1}  = D_{-1,-1} = D_{1,2} = 0$ on Killing vector fields.
\end{enumerate}
\end{ecor}
\proof
The first two statements are direct consequences of Proposition \ref{differential}.
Now, let $X$ be a Killing vector field. Then $X$ is coclosed and
$d^*d X = \Delta X = \frac{\scal}{2n}X$. Hence, the first equation   of 
 Proposition \ref{differential} gives $D_{-1,-1}X = 0$ and  adding the remaining three equations gives 
 \beq\label{1}
  \frac12  d d^* X \, = \, B_{-1,1}X + B_{1,-1}X  +  B_{1,2}X  \, =  \, \frac{\scal}{4n}X \.
  \eeq
 On the other hand, since  $\nabla^*\nabla X =  \frac{\scal}{4n}X$, the universal formula for $\nabla^*\nabla $ implies
 \beq\label{2}
 \nabla^*\nabla X \,=\,  B_{-1,1}X + B_{1,-1}X  +  B_{1,2}X   +  B_{-1,2}X  +  B_{1,1}X  \, =  \, \frac{\scal}{4n}X \ .
 \eeq
 Comparing \eqref{1}  and \eqref{2} proves the vanishing of $B_{-1,2}X $ and $ B_{1,1}X$. Finally we have
 $D_{1,2} X=0$, since the minimal eigenvalue on $\Sym^2 H \Lambda^2_0E$ is larger than $\lambda_2$. 
\qed

\medskip

\begin{ere}\label{killing-rem}
The only  generalised gradients not a priori vanishing on Killing vector fields  are $D_{-1,1}$ with values in $\Sym^2E$ and $D_{1,-1}$ with values
in $\Sym^2H$. The three Weitzenb\"ock formulas on $HE$ given above and Corollary \ref{closed}, (3) imply for Killing vector fields $X$ the equations
$$
B_{1,-1}X \,=\; \frac{3 \, \scal }{8n(n+2)} \, X
\qquad\mbox{and}\qquad
B_{-1,1}X \,=\; \frac{(2n+1) \, \scal }{8n(n+2)}  \, X \ .
$$
\end{ere}


Let  $(M^m, g)$ be a general Riemannian manifold. Denote with $\delta$ the divergence on symmetric tensors, i.e.
$\delta  : \Gamma(\Sym^p \T M) \rightarrow \Gamma(\Sym^{p-1} \T M)$ and with  $\delta^*$ its adjoint operator.  Recall 
that $\delta^* X = \frac12 L_Xg$ on vector fields and $\delta^* = d$ on functions. Similar to Proposition \ref{differential} we have

\begin{epr}\label{divergence}
Let  $(M^{4n}, g)$ be a quaternion-K\"ahler manifold, then on $ HE$ it holds that
$$
\delta^* \, \delta \;=\;4n \, B_{-1,-1}, \qquad
\delta \,  \pr_{\Sym^2H \Sym^2E}\, \delta^*  \;=\; 2 B_{1,1}, \qquad
\delta \, \pr_{\Lambda^2_0E } \, \delta^*  \;=\; 2  B_{-1,2}  \ .
$$
\end{epr}
\proof
On symmetric tensors we can consider generalised gradients 
$$
P_1 : \Gamma(\Sym^p_0 \T M) \rightarrow \Gamma(\Sym^{p+1}_0 \T M)
\qquad \mbox{and} \qquad
P_2 : \Gamma(\Sym^p_0 \T M) \rightarrow  \Gamma (\Sym^{p-1}_0 \T M)
$$
where here $\Sym^{p+1}_0 \T M$ and $\Sym^{p-1}_0 \T M$ are considered as subspaces of $\T M \otimes \Sym^{p}_0 \T M$.
Again $P_1$ and $P_2$ are defined as orthogonal projections of the covariant derivative onto these summands.
Then we have the equations (see \cite{HMS}, Lemma 2.3)
\beq\label{compare2}
\delta  \delta^* \;=\;  (p+1) \, P_1^* P_1
\qquad \mbox{and} \qquad
\delta^* \, \delta \;=\; \frac{(m+2p-2)(m+p-3)}{m+2p-4} \, P^*_2 P_2
\eeq
As in the case of forms the statement of the proposition is then a consequence of \eqref{compare2} for $p=1$
and $m=4n$, together with the decomposition $\Sym^2 \T M = \Sym^2H\Sym^2E \oplus \Lambda^2_0E \oplus \C$.
\qed

\medskip

\begin{ere}
There is also an alternative proof of Proposition \ref{divergence}. Indeed, since we have $d = \nabla = D_{1,1} : C^\infty(M) \rightarrow \Gamma(HE)$
the relative dimension formula implies  on sections of $HE$:
\beq\label{A}
\delta^* \delta \;=\; d d^*  \;=\;  D_{1,1} D_{1,1} ^*  \;=\;  \frac{\dim HE}{\dim \C} D^*_{-1,-1} D_{-1,-1}   \;=\;   4n B_{-1,-1}   \ .
\eeq
This proves again the first equation of Proposition \ref{divergence}.
On the other hand we can use the three Weitzenb\"ock formulas on $HE$ to rewrite the universal Weitzenb\"ock formula for the Laplace operator  on 
sections of $HE$ as
\beq\label{B}
\Delta \;= \; 2 B_{1,1} + 2 B_{-1, 2} \, +\,   2(1-2n) B_{-1,-1} \,   +  \, \frac{\scal}{2n} \; =\;  \delta \delta^* \,  -\,  \delta^* \delta \,   +  \, \frac{\scal}{2n} \ ,
\ee
where the second equality follows from the   Weitzenb\"ock formula   $\Delta = \delta \delta^* - \delta^* \delta + 2q(R)$ on symmetric $p$-tensors,
applied to a quaternion-K\"ahler manifold $(M^{4n}, g)$ and $p=1$.  Substituting \eqref{A}  for $\delta^*  \delta$   into \eqref{B} and solving for  $\delta \delta^* $ gives
$$
\delta  \delta^*  \;= \; 2 B_{1,1}  \, + \,  2 B_{-1, 2} \, +\,   2 B_{-1,-1} \ .
$$
\end{ere}

%
\section{Eigenvalue estimates on $2$-forms}
%

In this section we will derive improved eigenvalue estimates for small 
eigenvalues of  the Laplace operator $\Delta = d^*d + d d^*$ on $2$-forms. Moreover we describe the corresponding 
eigenspaces. On quaternion-K\"ahler manifolds there is the following decomposition of $\Lambda^2 \T M$
into parallel subbundles
\be\label{Lambda2}
\Lambda^2 \T M =  \Sym^2 H  \oplus \Sym^2 E \oplus \Sym^2H\Lambda^2_0E \ .
\ee
It is known that only $\Sym^2 E \cong \Lambda^{1,1}_0E$ can carry harmonic forms \cite{SW2002}. Hence,
on this subbundle we have zero as the lower bound of the Laplace spectrum.  However, by \cite{LS}
this is only possible for the complex Grassmannian $\mathrm{Gr}_2(\C^{n+2})$, where
the harmonic $2$-forms are constant multiples of the K\"ahler form. We will give a new estimate
for the first non-vanishing eigenvalue.

%
\subsection{Eigenvalue estimates on $\Sym^2E$} \label{sym2}
%

On this space there are three generalised gradients. Restricted to the $\lambda$-eigenspaces we have
by the commutator rule of Proposition \ref{commutator}.
$$
D_{1,-1} : (\Sym^2E )_\lambda \rightarrow (HE)_\lambda, \;\;
D_{1,2} : (\Sym^2E )_\lambda \rightarrow (H\Lambda^{2,1}_0E)_\lambda, \;\;
D_{1,1} : (\Sym^2E )_\lambda \rightarrow \Gamma(H\Sym^3 E)
$$
We do not have a commutator rule for $D_{1,1}$ since the bundle $H\Sym^3 E$ is not of the
form necessary for the commutator formula.
Recall that the minimal eigenvalue on $HE$ is $\frac{n+1}{n+2} 2\E$ and on $H\Lambda^{2,1}_0E$
it is $\frac{n}{n+2} 2\E$. On $\Sym^2E $ there is only one Weitzenb\"ock formula. However, it involves
the hyper-K\"ahler curvature, so we cannot use it directly. We have
$$
q(R) = - B_{1,1} + \frac12 B_{1,2} + (n+1) B_{1,-1} \ .
$$
Hence, the Weitzenb\"ock formula for $\Delta$ and the universal formula for $\nabla^*\nabla$ imply
\beq\label{laplace1}
\Delta \;=\; \nabla^*\nabla  \,+ \,  q(R) \;=\; \frac32 B_{1, \, 2} \,+\, (n+2) B_{1,-1} \ .
\eeq
Here $\Delta = \Delta_{\Sym^2 E}$ but since $\Sym^2 E \subset \Lambda^2 \T M$ is a parallel
subbundle, the operator $\Delta_{\Sym^2 E}$ coincides with the Hodge-Laplacian $\Delta = d d^* + d^* d$
restricted to sections of $\Sym^2 E$.

The following proposition describes the relation between differential, codifferential and 
generalised gradients on sections of $\Sym^2E \subset \Lambda^2 \T M$. 

\begin{epr}\label{harmonic2}
Let $(M^{4n}, g)$ be a quaternion-K\"ahler manifold. Then restricted to sections of   $\Sym^2 E \subset \Lambda^2 \T M$ 
the following equations hold
$$
\pr_{\Sym^2 E} \,d  \, d^*  \;=\; \frac{2n+1}{2} \, B_{1,-1},
\;\;
 \pr_{\Sym^2 E} \,  d^* \, \pr_{HE} \, d 	\;=\; \frac{3}{2} \, B_{1,-1} ,
\;\;
\pr_{\Sym^2 E} \,d^* \,  \pr_{H\Lambda^{2,1}_0E} \, d\;=\;  \frac32 \, B_{1,\, 2} \ .
$$
\end{epr}
\proof
On sections of $HE$ we have $\pr_{\Sym^2E} \circ d  = c D^*_{1,-1} $ for some complex number $c$. Thus,
$d^* \pr_{\Sym^2E} d = (\pr_{\Sym^2E} \circ d)^* (\pr_{\Sym^2E} \circ d) = |c|^2 D_{1,-1}D^*_{1,-1}$.
On the other side, using Proposition \ref{differential} and the relative dimension formula,  we obtain on sections of $HE$:
$$
d^* \pr_{\Sym^2E} d \,= \,2 D^*_{-1,1} D_{-1,1} \, = \, 2 \frac{\dim \Sym^2 E }{\dim HE} D_{1,-1} D^*_{1,-1} \, =\,  \frac{2n+1}{2} D_{1,-1} D^*_{1,-1} \ .
$$
It follows $ |c|^2 =  \frac{2n+1}{2} $. With $\pr_{\Sym^2E}^* = \iota_{\Sym^2E} : \Sym^2 E \hookrightarrow \Lambda^2 \T M$ we find
$$ 
\pr_{\Sym^2E} \, d \, d^* \,  \iota_{\Sym^2E}  = (\pr_{\Sym^2E} \circ d)  (\pr_{\Sym^2E} \circ d)^* =  |c|^2 D_{1,-1}^*D_{1,-1} = \frac{2n+1}{2} B_{1,-1} \ .
$$
Next we rewrite \eqref{laplace1} as:  $\Delta = ( \frac32 B_{1, \, 2} \,+\, \frac32  B_{1,-1}) +  \frac{2n+1}{2} B_{1,-1}  $. Since
$\Delta$ preserves the subbundle $\Sym^2 E$ we can write 
$\Delta = \pr_{\Sym^2E} \, \Delta \, \iota_{\Sym^2E}  = \pr_{\Sym^2E} \,  d^* d  \, \iota_{\Sym^2E}  + \pr_{\Sym^2E}  \, d d^* \,\iota_{\Sym^2E} $. Hence,
\begin{equation}\label{sym2E}
\pr_{\Sym^2E} \, d^*  \, d \,  \iota_{\Sym^2E}  =  \frac32 B_{1, \, 2} \,+\, \frac32  B_{1,-1} \ .
\end{equation}
We have $d \, \iota_{\Sym^2 E} = \pr_{HE} \, d  \, \iota_{\Sym^2 E}  +  \pr_{H\Lambda^{2,1} E} \, d  \, \iota_{\Sym^2 E} $. Then
$  \pr_{HE} \, d  \, \iota_{\Sym^2 E}   =  c_1 D_{1,-1}$ and we conclude  $  \pr_{H\Lambda^{2,1} E} \, d  \, \iota_{\Sym^2 E}  = c_2 D_{1,2}$ for some complex numbers $c_1, c_2$. From here we obtain
$$
\begin{array}{ccl}
\pr_{\Sym^2E} \, d^*  \, d \,  \iota_{\Sym^2E}  & = & (d \, \iota_{\Sym^2 E}   )^* (d \, \iota_{\Sym^2 E}  ) \\[1ex]
   &=& 
( \pr_{HE} \, d  \, \iota_{\Sym^2 E} )^* (\pr_{HE} \, d  \, \iota_{\Sym^2 E})  +  (\pr_{H\Lambda^{2,1} E} \, d)^* \, 
(\pr_{H\Lambda^{2,1} E} \, d  \, \iota_{\Sym^2 E} )\\[1ex]
&=& |c_1|^2 B_{1,-1} +  |c_2|^2 B_{1,2} \ .
\end{array}
$$
The operators $ B_{1,-1}$ and $B_{1,2}$ are linearly independent. Indeed, otherwise one could rewrite \eqref{laplace1}
as $\Delta = aB_{1,-1} $ for some constant $a \neq 0$. But the first equation of Proposition \ref{harmonic2}, which we already
proved, together with the Hodge decomposition theorem for $2$-forms would imply the vanishing of $ B_{1,-1}$ and also of $B_{1,2}$.
Hence, the last equation together with \eqref{sym2E} yields $ |c_1|^2 =  |c_1|^2 = \frac32$. Then 
$$
 \pr_{\Sym^2 E} \,  d^* \, \pr_{HE} \, d \,  \iota_{\Sym^2E}  = ( \pr_{HE} \, d \,  \iota_{\Sym^2E}  )^* ( \pr_{HE} \, d \,  \iota_{\Sym^2E}  )
 = |c_1|^2 B_{1,-1}   = \frac32 B_{1,-1} \ .
$$
The last equation of Proposition \ref{harmonic2} follows similarly.
\qed

\medskip

\begin{ecor}\label{harmonic3}
Let $(M^{4n}, g)$ be a positive quaternion-K\"ahler manifold. Then any closed form in $\Sym^2 E$ is
harmonic. Moreover, a form $\varphi$ in  $\Sym^2 E$ is harmonic if and only if it satisfies
$B_{1,1} \varphi = -q(R) \varphi$, where $q(R)$ on $\Sym^2 E$ is written as $q(R) = \frac{n+1}{n+2} 2\E \, \id + q(R^{hyper})$.
\end{ecor}
\proof
Let $\varphi$ be a closed $2$-form in $\Sym^2E$. 
Then the second and third equation of Proposition \ref{harmonic2} imply $D_{1,-1} \, \varphi = 0 = D_{1,2} \, \varphi$. Thus  by equation \eqref{laplace1}
the form $\varphi $ is  harmonic.

Let $\varphi$ be a harmonic form then $\varphi$ is closed and coclosed. Hence  $D_{1,-1}\, \varphi = 0 = D_{1,2}\,  \varphi$ by
Propostion \ref{harmonic2}. But then $0 = \Delta \varphi = \nabla^*\nabla \varphi + q(R) \varphi = B_{1,1} \varphi + q(R) \varphi$
by the universal formula for $ \nabla^*\nabla$.
Thus $B_{1,1} \varphi = -  q(R) \varphi$. Conversely, $B_{1,1} \varphi = -  q(R) \varphi$ implies
$\Delta \varphi = B_{1,-1} \varphi + B_{1,2} \varphi$. Subtracting \eqref{laplace1} and recalling that $ B_{1,-1} $ and $ B_{1,2} $
are non-negative, immediately gives $ B_{1,-1} \varphi  = 0 =  B_{1,2}\varphi $, which shows that $\varphi$ has to be harmonic.
The formula for $q(R)$ can be found in \cite{SW2002}.
\qed

\medskip

Next we give an estimate for the Laplace spectrum on  non-harmonic forms in $\Sym^2E \cong \Lambda^{1,1}_0E$.

\begin{epr}\label{improved}
Let $(M^{4n}, g)$ be a positive quaternion-K\"ahler manifold. Then the first non-vanishing $\Delta$-eigenvalue $\lambda$
on $\Sym^2 E$ satisfies:
$
\lambda \ge \frac{\scal}{2(n+2)}
$.
\end{epr}
\proof
Consider a section $\varphi \in \Gamma(\Sym^2E)$ with $\Delta \varphi = \lambda \varphi$ and $\lambda \neq 0$. Assume that
$\lambda $ is smaller than $ \frac{\scal}{2(n+2)}$, which is the minimal eigenvalue of  $H\Lambda^{2,1}_0E$. But then
$\lambda$ is also smaller than $\frac{n+1}{n+2}\frac{\scal}{2n}$, which is  the minimal eigenvalue of $HE$. Hence, 
$D_{1,2}\varphi = 0 = D_{1,-1} \varphi$ and in consequence $\varphi $ is harmonic by \eqref{laplace1}, contradicting the 
assumption $\lambda \neq 0$.
%
%
%
%
\qed

\medskip

\begin{ere}
Using standard techniques from harmonic analysis one can calculate the spectrum of the Hodge-Laplace operator $\Delta$ on $\Lambda^{1,1}_0E \cong
\Sym^2  E \subset \Lambda^2 \T M$ for the the symmetric space $\mathrm G_2/\SO(4)$ (the relevant branching rules can be found in \cite{Se99}). 
It turns out that in this case $\frac{\scal}{2(n+2)}$ is the smallest
eigenvalue of $\Delta$, i.e. our improved estimate is sharp. Moreover, the multiplicity of the smallest eigenvalue can be
seen to be $7$. Alternatively one can use a calculation of the index
$i^{2,n}$ to see that the equality case in our estimate holds for $\mathrm G_2/\SO(4)$ (see  Appendix A  for more details).
\end{ere}

\medskip

Using similar methods we can still prove  two further remarks on  small eigenvalues.

\begin{epr}\label{injective1}
Let $(M^{4n}, g)$ be a positive quaternion-K\"ahler manifold. 
\begin{enumerate}
\item 
For any  $\lambda$ with $\frac{\scal}{2(n+2)} \le \lambda < 2\E$ there is an injective map
$(\Sym^2 E)_\lambda \rightarrow (H\Lambda^{2,1}_0E)_\lambda$.
\medskip
\item
There is an injective map $\, \iso(M, g) \rightarrow (\Sym^2 E)_{2\E}$.
\end{enumerate}
\end{epr}
\proof
We claim that the injective map in (1) is given by $D_{1,2}$. Indeed, let $\varphi$ be a $2$-form 
in $(\Sym^2 E)_\lambda $. Then $X := d^* \varphi$ is a coclosed vector field
in $(HE)_\lambda$. But, since $\lambda < 2\E$ this is only possible with $X=0$ , i.e. $\varphi$ is coclosed. 
Then the first equation of Proposition \ref{harmonic2} implies $B_{1,-1} \varphi =0$
and by  \eqref{laplace1} we have $\lambda \varphi = \Delta \varphi =  \frac32 B_{1,2} \,\varphi = \frac32 D_{1,2}^* D_{1,2} \varphi$. Hence $D_{1,2}$ is injective, since $\varphi$ is not harmonic.

%
%

\noindent
In  (2) we claim that the injective map is given by $D := \pr_{\Sym^2E} \, d$ (or alternatively by $D_{-1,1}$). Indeed, let $X$ be a Killing vector field then $X \in (HE)_{2\E}$.
Moreover  we have $d^* D X = 2 B_{-1,1}X = \frac{2n+1}{4n(n+2)} \scal \,X$ by Proposition \ref{differential} and Remark \ref{killing-rem}. Thus $D$ is injective.
\qed

\begin{ere}
It can be shown that  in the extremal case $\lambda = \lambda_3:= \frac{\scal}{2(n+2)} $ the generalised gradient $D_{-1,1}$ defines an isomorphism $(\Sym^2 E)_{\lambda_3} \rightarrow (H\Lambda^{2,1}_0E)_{\lambda_3}$ (see Appendix A for more details).
\end{ere}

%
\subsection{Eigenvalue estimates on \texorpdfstring{$\Sym^2 H$}{\Sym^2 H}   }
%

The minimal eigenvalue of the representation $\Sym^2 H$ is $\lambda_2 = 2\E$. It is well known
that the corresponding minimal eigenspace is isomorphic to the space of 
Killing vector fields. We will recall the argument and give further information
on this eigenspace, together with a lower bound for all eigenvalues different from $2\E$.

On $\Sym^2H$ we have two generalised gradients and one Weitzenb\"ock formula. Restricted
to $\lambda$-eigenspaces  and using the commutator rule we have
$$
D_{1,1} : (\Sym^2 H)_\lambda \rightarrow ( \Sym^3HE)_\lambda
\qquad\mbox{and}\qquad
D_{-1,1} : (\Sym^2 H)_\lambda \rightarrow (HE)_\lambda
$$
Note, that both,  $\Sym^3HE$ and $HE$ are summands of $\Lambda^3 \T M$ and recall that
$q(R) = \frac{\scal}{n(n+2)}$ holds on $\Sym^2 H$. The
Weitzenb\"ock formula on  $\Sym^2H$ is given as
\beq\label{wbf2}
- B_{1,1} \, +  \, 2 B_{-1,1} \,=\,  \frac{\scal}{n+2} \ .
\eeq
The universal  formula for $\nabla^*\nabla$ and \eqref{wbf2} immediately  lead to
\beq\label{wbf3}
\Delta \;= \;  \nabla^*\nabla \, + \, q(R)   \;= \; B_{-1,1}  \,+\,   B_{1,1}  \, + \, \frac{\scal}{n(n+2)}
\; =\;
\frac32 B_{1,1}  \, + \, \frac{\scal}{2n} \ .
\eeq
This proves $\lambda \ge 2\E$ and shows that $(\Sym^2 H)_{2\E}  = \ker D_{1,1}$. Finally, we
want to show that $\ker D_{1,1}$ is isomorphic to the space of Killing vector fields. For this we still
need the following

\begin{epr}\label{lemma}
On sections of $\Sym^2H$ it holds that
$$
\pr_{\Sym^2H}  \,d \, d^* \, =\, \frac{3}{2n} \, B_{-1,1}, \qquad
d^* \, \pr_{HE} d \, =\, \frac{2n+1}{2n} \,  \, B_{-1,1} , \qquad
d^* \, \pr_{\Sym^3HE} d \, =\, \frac{n-1}{n} \,  \, B_{1,1}  \ .
$$
\end{epr}
\proof
On sections of $HE$ we have $\pr_{\Sym^2H} \circ d = c D_{-1,1}^*$ for
some complex number $c$ and
$$
 |c|^2 D_{-1,1} D_{-1,1}^* = (\pr_{\Sym^2H} \circ d )^*(\pr_{\Sym^2H} \circ d ) = 2  D_{1,-1}^* D_{1,-1}
 = \frac{3}{2n} D_{-1,1} D_{-1,1}^* \ .
$$
by Proposition \ref{differential} and the relative dimension formula. It follows $|c|^2 =  \frac{3}{2n} $ and thus also
$$
\pr_{\Sym^2H}  \,d \, d^* = (\pr_{\Sym^2H} \circ d )  \, (\pr_{\Sym^2H} \circ d )^* = |c|^2  D_{-1,1}^* D_{-1,1} =  \frac{3}{2n}  B_{-1,1} \ ,
$$
which proves the first equation of the proposition. Next, as in the proof of Proposition \ref{harmonic2}, we rewrite the 
Laplacian $\Delta$, where we replace $\frac{\scal}{n(n+2)}$ in \eqref{wbf3} using \eqref{wbf2} and thus obtain
$$
\Delta = d^* d + d d^* =  \frac{n-1}{n} B_{1,1} + \frac{n+2}{n} B_{-1,1} = \Big(  \frac{n-1}{n} B_{1,1} + \frac{2n+1}{2n} B_{-1,1} \Big)  + \frac{3}{2n} B_{-1,1}  \ .
$$
Substituting the first equation of Proposition \ref{lemma}, which we already proved, yields 
$$
\pr_{\Sym^2H}  \,d^* \, d =  \frac{n-1}{n} B_{1,1} + \frac{2n+1}{2n} B_{-1,1} \ .
$$
The last part of the proof is completely analogous to that of Proposition \ref{harmonic2}. This time using the linear independence of the operators  $B_{1,1} $ and $B_{-1,1}$.
\qed

\medskip

The next corollary is  a first consequence of Proposition \ref{lemma}.

\begin{ecor}\label{injective}
Any $2$-form in $\Sym^2 H \subset \Lambda^2 \T M$ which is either closed or coclosed vanishes.
\end{ecor}
\proof
First, let $\varphi$ be a closed  $2$-form in  $\Sym^2 H $. Then $D_{-1,1} \varphi = 0 = D_{1,1} \varphi$ by Proposition \ref{lemma}. 
Hence, applied to $\varphi$ the left side of \eqref{wbf2} vanishes and thus also $\varphi$ has to vanish.
Next, let $\varphi$ be a coclosed  $2$-form in  $\Sym^2 H $. Then the first equation of Proposition \ref{lemma} implies $D_{-1,1} \varphi =0$ and after integration in \eqref{wbf2}
we see that $ D_{1,1} \varphi = 0$. It follows as in the first case that  $\varphi$  has to vanish.
\qed

\medskip

As an other  consequence of Proposition \ref{lemma} and   the argumentation above we obtain

\begin{epr}
Let  $(M^{4n}, g)$  be a positive quaternion-K\"ahler manifold. Then any eigenvalue
$\lambda $ of the Laplace operator on sections of $\Sym^2 H \subset \Lambda^2 \T M$ satisfies
$\lambda \ge 2\E$. Moreover, the  eigenspace for the lower bound $2\E$ is isomorphic to the space of Killing vector fields,
i.e. $(\Sym^2 H)_{2\E} \cong \iso (M, g)$.
\end{epr}
\proof
It remains to prove the isomorphism $(\Sym^2 H)_{2\E} \cong \iso (M, g)$.
We claim that it is given by the codifferential. Consider  
$\omega \in \ker D_{1,1} = (\Sym^2H)_{2\E}$, with $\omega \neq 0$. Then $X := d^* \omega \in (HE)_{2\E}$ 
is a Killing vector field, since it is coclosed and an eigenvector for the eigenvalue $2\E$. Hence the
map $d^*: (\Sym^2H)_{2\E} \rightarrow \iso(M,g)$ is well-defined. It is injective by Corollary \ref{injective}.

Conversely let $X$ be a Killing vector field. Then $d^* \pr_{\Sym^2 H} \, d X = 2 B_{-1,1} X = \frac{3\scal}{8n(n+2)}  X$,
by Proposition \ref{differential} and Remark \ref{killing-rem}. Hence, $d^*$
is  surjective and thus defines an isomorphism.
\qed


\begin{ere}
A map from $\iso(M, g) $ to $(\Sym^2H)_{2\E}$ can also be described as follows: take a Killing vector field $X$.
Then $dX = 2 \nabla X$ is a section of the Lie algebra of the holonomy group, i.e. it lies in $\Sym^2 H \oplus \Sym^2E$
and thus  $ \pr_{\Sym^2 H} (dX) \in (\Sym^2H)_{2\E}$. The projection of $dX$ to $\Sym^2E$ defines an injective map 
$\iso(M,g) \rightarrow (\Sym^2E)_{2\E}$ (see Proposition \ref{injective1}, (2)).
\end{ere}


Finally we give a  lower bound for the eigenvalues different from $2\E$.

\begin{epr}
Let  $(M^{4n}, g)$  be a positive quaternion-K\"ahler manifold. Then for any eigenvalue $\lambda$ of $\Delta$
on the orthogonal complement of $(\Sym^2H)_{2\E}$ in $\Gamma(\Sym^2 H)$ it holds that $\lambda \ge 4\E$. In particular,
this estimate is true if $(M, g)$ does not admit non-trivial Killing vector fields.
\end{epr}
\proof
Assume that $\omega \in (\Sym^2H)_\lambda$ with $2\E < \lambda < 4\E$.
Then $D_{1,1}\omega$ is a non-vanishing section of $(\Sym^3HE)_\lambda$. But the minimal eigenvalue of $\Sym^3HE$
is $4\E$.  Hence, it follows $\lambda \ge 4\E$.
\qed


%
\subsection{Eigenvalue estimates on   \texorpdfstring{$\Sym^2H \Lambda^2_0E$}{\Sym^2H \Lambda^2_0E}  }
%
On this last summand of $\Lambda^2 \T M$ the minimal eigenvalue is $\frac{2(n+1)}{n+2} 2\E$, which
is larger than the minimal non-vanishing eigenvalues on the two other summands $\Sym^2E$ and $\Sym^2H$.
Indeed we have $\frac{2(n+1)}{n+2} 2\E> 2\E = \frac{\scal}{2n}  > \frac{\scal}{2(n+2)}$. Hence we can summarise the estimates
in this section into the following final statement.

\begin{epr}
Let  $(M^{4n}, g)$  be a positive quaternion-K\"ahler manifold. Then any non-vanishing eigenvalue $\lambda$ of the
Laplace operator on $2$-forms satisfies $\lambda \ge  \frac{\scal}{2(n+2)}$.
\end{epr}


\begin{ere}
Our result improves an estimate given in \cite{AMP}, Thm. 5.4. Translated into our notation
they prove $\lambda \ge 2\E$ on $\Sym^2H \oplus \Sym^2H\Lambda^2_0E$, which is optimal
only on $\Sym^2H $.
\end{ere}

%
\section{Eigenvalue estimates on symmetric   \texorpdfstring{$2$}{2}-tensors}
%

In this section we want to study eigenvalue estimates for the Lichnerowicz-Laplacian $\Delta$
on trace-free symmetric $2$-tensors. In particular, we want to describe eigenspaces
for eigenvalues below, or equal to,  the critical eigenvalue $\lambda_2 = 2 \E$. Recall that
the bundle of trace-free symmetric $2$-tensors has the following decomposition into parallel
subbundles:
\be\label{Sym2}
\Sym^2_0 \T M  \,=\,  \Lambda^2_0E \,  \oplus \,\Sym^2H\Sym^2E \ .
\ee
Hence, any section $h$ of $\Sym^2_0 \T M$ can be written as
a sum $h=h_1 + h_2$. If $\Delta h = \lambda h$ holds then  $h_1$ and $h_2$
are sections of  $(\Lambda^2_0E)_\lambda$ and $ (\Sym^2H\Sym^2E)_\lambda$,
respectively. Both representations have the same minimal eigenvalue $\lambda_1 = \frac{n+1}{n+2} 2\E$.
So we know already  $\lambda \ge \lambda_1$.

%
\subsection{Eigenvalue estimates on   \texorpdfstring{$\Lambda^2_0E$}{\Lambda^2_0E}    }
%

On $\Lambda^2_0E$ we have  three generalised gradients. Restricting to  eigenspaces for the eigenvalue $\lambda$ 
and using  the commutator relation as above gives
$$
D_{1,1}   :   (\Lambda^2_0E)_\lambda  \longrightarrow (H\Lambda^{2,1}_0E)_\lambda, \quad
D_{1,3}   :   (\Lambda^2_0E)_\lambda  \longrightarrow (H\Lambda^{3}_0E)_\lambda, \quad
D_{1,-2}   :   (\Lambda^2_0E)_\lambda  \longrightarrow (H E)_\lambda 
$$
Note that $q(R) = \frac{\scal}{2(n+2)}$ on $\Lambda^2_0E$. The only Weitzenb\"ock formula on  $\Lambda^2_0E$ is
\beq\label{wbf4}
2 B_{1,3} \, - \, B_{1,1}  \;+ \; 2n \, B_{1, -2 } \;=\; \frac{\scal}{n+2} \ .
\eeq
The universal Weitzenb\"ock formula for  $\Delta$ together with \eqref{wbf4} gives
\beq\label{wbf5}
\Delta = B_{1,1} + B_{1,3} + B_{1,-2} +  \frac{\scal}{2(n+2)} 
= \frac32  B_{1,1} - (n-1) B_{1, -2 } +  \frac{\scal}{n+2}   \ .
\eeq
In the present situation the standard Laplace operator  $\Delta$ coincides with the Lichnerowicz\-Laplacian $\Delta_L$ on symmetric $2$-tensors
and we have the additional well-known Weitzenb\"ock formula
\beq\label{wbf6}
\Delta = \delta \, \delta^* - \delta^* \delta + 2 q(R) \ .
\eeq

\medskip

The following proposition gives the relation between the divergence $\delta$ and its adjoint operator $\delta^*$ to the generalised gradients of $\Lambda^2_0E$.

\begin{epr}\label{divergence2a}
Let $(M^{4n}, g)$ be a quaternion-K\"ahler manifold. Then on $\Lambda^2_0E$  it holds that
$$
\pr_{\Lambda^2_0E} \, \delta \delta^* \;=\; \frac32 B_{1,1} \,+\, \frac{n-1}{2n} B_{1,-2}
\qquad\mbox{and}\qquad
\pr_{\Lambda^2_0E} \, \delta^* \delta \;=\;  \frac{(2n+1)(n-1)}{2n} B_{1,-2} 
$$
\end{epr}
\proof
On sections of $HE$ we have $\pr_{\Lambda^2_0E} \circ \delta^* = c D^*_{1,-2}$ for some complex number $c$. Then
$$
\begin{array}{ccl}
 |c|^2 D_{1,-2} \, D^*_{1,-2} & =& (\pr_{\Lambda^2_0E} \circ \delta^*)^*  \, (\pr_{\Lambda^2_0E} \circ \delta^*) =
 \delta \,  \pr_{\Lambda^2_0E}  \,  \delta^* = 2 D^*_{-1,2} \,D_{-1,2} \\[1ex]
 &  =&  \frac{(2n+1)(n-1)}{2n} D_{1,-2} \, D_{1,-2}^* \, .
\end{array}
$$
by Proposition \ref{divergence} and the relative dimension formula. Hence, $ |c|^2 = \frac{(2n+1)(n-1)}{2n} $ and we see
$$
\pr_{\Lambda^2_0E} \, \delta^* \delta = 
(\pr_{\Lambda^2_0E} \circ \delta^*) \, (\pr_{\Lambda^2_0E} \circ \delta^*)^* = |c|^2 D^*_{1,-2} \, D_{1,-2} =
\frac{(2n+1)(n-1)}{2n} B_{1,-2} \ ,
$$
which proves the second equation of Proposition \ref{divergence2a}.
Using \eqref{wbf6} we rewrite \eqref{wbf5} as 
$$
\Delta 
= \delta \, \delta^* - \delta^* \delta +  \frac{\scal}{n+2}  
= (\frac32  B_{1,1}  + \frac{n-1}{2n} B_{1, -2 }) - \frac{(2n+1)(n-1)}{2n} B_{1, -2 } +  \frac{\scal}{n+2}   \ .
$$
Substituting the formula for $\pr_{\Lambda^2_0E} \, \delta^* \delta$ proves the first equation of Proposition \ref{divergence2a}.
\qed

\medskip

As a direct application of the last proposition we have the following statement.

\begin{ecor}\label{h2}
Let $(M^{4n}, g)$ be a positive quaternion-K\"ahler manifold with $n>2$.
For any $h \in (\Lambda^2_0E)_{\lambda}$ with $\lambda \le \lambda_2 = 2\E$  and $\delta h = 0$ it follows that $h=0$.
\end{ecor}
\proof
Since $\delta h = 0$ implies  $D_{1,-2}h=0$ by Proposition \ref{divergence2a}, the statement follows from the injectivity of the map $D_{1,-2}$
on $ (\Lambda^2_0E)_{\lambda}$ for $\lambda \le \lambda_2$. Indeed, if $\lambda \le \lambda_2$ we get 
$D_{1,3} h = 0$, since the minimal eigenvalue of $H\Lambda^3_0E$ is $\frac{\scal}{n+2}$, which is larger than $\lambda_2$ for $n>2$. If $D_{1,-2}h=0$,
then \eqref{wbf4} yields $-B_{1,1} h = \frac{\scal}{n+2} h$. Hence,   $D_{1,1}h= 0 $ by integration, i.e. $B_{1,1} h=0$ and
thus \eqref{wbf4} implies $h=0$.
\qed

\medskip

%

The following proposition describes the eigenspaces of $\Delta$ on sections of $\Lambda^2_0E \subset \Sym^2 \T M$, 
for eigenvalues between the lower bound  $\lambda_1$ and the critical eigenvalue $\lambda_2$.

\begin{epr}\label{LambdaE}
Let $(M^{4n}, g)$ be a positive quaternion-K\"ahler manifold with $n>2$. Then for any eigenvalue
$\lambda$ with $\lambda_1 \le \lambda \le \lambda_2$  it holds that \, $(\Lambda^2_0E)_\lambda \cong C^\infty(M)_\lambda$.
\end{epr}
\proof
We claim that for $\lambda < \lambda_2$ the operator $D_{1,-2} :  (\Lambda^2_0E)_\lambda  \rightarrow (HE)_\lambda$
defines an isomorphism. Under this condition on $\lambda$ we have  $ (HE)_\lambda \cong  C^\infty(M)_\lambda $.

We have  seen in the proof of Corollary \ref{h2} that for $\lambda \le \lambda_2$ the map $D_{1,-2}$ is injective
on $ (\Lambda^2_0E)_\lambda $. 

So it remains to show the surjectivity of $D_{1,-2}$ we take $X \in (HE)_\lambda$ with $X \neq 0$. As remarked after \eqref{eq1}
we have $D_{-1,2}X \neq 0$ and the relative dimension formula gives:
\beq\label{reldim}
D_{1,-2} \, D_{1,-2}^* X \;=\; \frac{\dim HE}{\dim \Lambda^2_0E} \, D_{-1,2}^* \, D_{-1,2} \, X 
\;=\;  \frac{\dim HE}{\dim \Lambda^2_0E} \, \frac{n-1}{4n}  \left( \lambda \,+\, \frac{\scal}{2(n+2)} \right) \, X \ ,
\eeq
where the last equality follows from \eqref{eq3}. We see that $X$ is indeed in the image of $D_{1,-2}$.

Next, we consider the map $D_{1,-2}$ in the case $\lambda = \lambda_2$. From the argument above we already know
that it is injective. Recall that $ (HE)_{\lambda_2} = \iso(M,g) \oplus d \, C^\infty(M)_{\lambda_2}$. We note that
$\im ( D_{1,-2} )$ is orthogonal to the space of Killing vector fields $\iso(M,g)$. Thus the image of $D_{1,-2} $ is a subspace of $d \, C^\infty(M)_{\lambda_2}$. Indeed, 
$$
\la D_{1,-2} h, X\ra_{L^2} \; = \;\la  h, D_{1,-2}^* X\ra_{L^2} \; = \; 0 \ .
$$
The last equality follows since $D_{1,-2}^* X = 0$ for a Killing vector field $X$. This is, after integration, a consequence of the
relative dimension formula used in \eqref{reldim} and the vanishing of $D_{-1,2} X$ as shown in  Corollary \ref{closed}, (3).
The statement just proved, corresponds to the well-known statement $\im \, \delta \perp \ker \delta^*$ on general Riemannian
manifolds.

Finally, we still have  to show that $D_{1,-2} :  (\Lambda^2_0E)_{\lambda_2}  \rightarrow  d \, C^\infty(M)_{\lambda_2} $ is surjective.
We had seen that \eqref{eq3} also holds for vector fields $X = df$. Thus also  \eqref{reldim} and the 
surjectivity of $D_{1,-2} $ follows.
\qed


\begin{ere}
The identification between  $ (\Lambda^2_0E)_{\lambda} \subset \Gamma(\Sym^2 \T M)$ and $d C^\infty(M)_\lambda$ for  $\lambda \le \lambda_2$ can also
be defined by the divergence $\delta$. In fact Corollary \ref{h2} shows the injectivity of $\delta$ and the proof of Proposition \ref{LambdaE} together 
with Proposition \ref{divergence} shows the surjectivity.
\end{ere}

\begin{ere}
We will not study the case of quaternionic dimension $n=2$ left out in Corollary \ref{h2} and Proposition \ref{LambdaE}. 
In this dimension it is known by  \cite{PS}  that the only positive quaternion 
K\"ahler manifolds are $\H P^2, \, \mathrm{Gr}_2(\CM^{4})$ and $\mathrm G_2/\SO(4)$. They are all stable due to  \cite{Koiso} and \cite{SW2022}.
\end{ere}



%
\subsection{Eigenvalue estimates on  \texorpdfstring{$\Sym^2H\Sym^2E$}{\Sym^2H\Sym^2E}}
%
On the representation $V:=\Sym^2H\Sym^2E$ there are 6 generalised gradients. Restricted to the $\lambda$-eigenspaces and 
using the commutator relation we have
$$
 \begin{array}{llcllc}
     D_{-1,1} :    & V_\lambda  \longrightarrow & \Gamma(H \Sym^3 E)            & \qquad  D_{1,1} :       & V_\lambda  \longrightarrow & \Gamma(\Sym^3H\Sym^3 E) \\
    D_{-1,2}  :   & V_\lambda  \longrightarrow &(H \Lambda^{2,1}_0 E)_\lambda    &  \qquad D_{1,2} :    & V_\lambda  \longrightarrow & (\Sym^3 H\Lambda^{2,1}_0E)_\lambda    \\
      D_{-1,-1} : & V_\lambda  \longrightarrow  &(HE)_\lambda            &  \qquad D_{1,-1} :    & V_\lambda  \longrightarrow & (\Sym^3 HE)_\lambda  
\end{array} 
$$ 
We do not have a commutator formula for $ D_{-1,1}$ and $D_{1,1} $  for  $V:=\Sym^2H\Sym^2E$ .

On $\Sym^2H\Sym^2E$ there are three Weitzenb\"ock formulas. The universal Weitzenb\"ock formula for $q(R)$   involves the hyper-K\"ahler curvature $R^{hyper}$, as part of $q(R)$. 
\[
\begin{aligned}
\frac{1}{n+2}\,\scal
&= - B_{1,1}
   + 2 B_{-1,1}
   + 2 B_{-1,2}
   + 2 B_{-1,-1}
   - B_{1,2}
   - B_{1,-1}, \\[1ex]
\frac{2(n+1)}{n(n+2)}\,\scal
&= - (n+3) B_{1,1}
   + 2(n+3) B_{-1,1}
   - n B_{-1,2}
   + 2 (n+1)^2 B_{-1,-1} 
   + \frac{n}{2} B_{1,2}
   \\[.5ex]
   & \phantom{xxxxxxxxxxxxxxxxxxxxxxxxxxxxxxxxxxxxxxxxxxx}
    - (n+1)^2 B_{1,-1} \\
q(R) &=     -2 B_{1,1}
   +  B_{-1,1}
   - \frac12 B_{-1,2}
   + (3n+1) B_{-1,-1}
   + B_{1,2}
   + B_{1,-1},
\end{aligned}
\]

The following proposition gives the relation between the divergence $\delta$ and its adjoint operator $\delta^*$ to the generalised gradients of  $\Sym^2H\Sym^2E$.

\begin{epr}\label{final}
On sections of  $\Sym^2H\Sym^2E$ the following equations hold
$$
\pr_{\Sym^2H\Sym^2E} \, \delta \delta^* \,=\, \frac32 B_{-1,2} \,+\, 3 \B_{1,1} \, + \, \frac{3}{2} B_{-1,-1}
\quad\mbox{and}\quad
\pr_{\Sym^2H\Sym^2E} \, \delta^* \delta \,=\,  \frac{3(2n+1)}{2} B_{-1,-1}
$$
\end{epr}
\proof
On sections of $HE$ we have $\pr_{\Sym^2H\Sym^2E} \circ \delta^* = c D_{-1,-1}^*$ for some 
complex number $c$. Then by Proposition \ref{divergence} and the relative dimension formula we find
$$
\begin{array}{ccl}
|c|^2 D_{-1,-1} \, D_{-1,-1}^* &=& (\pr_{\Sym^2H\Sym^2E} \circ \delta^*)^*
(\pr_{\Sym^2H\Sym^2E} \circ \delta^*) = 2 D^*_{1,1} \, D_{1,1} \\[1.5ex]
&= &\tfrac{3(2n+1)}{2}  D_{-1,-1} \, D_{-1,-1}^*
\end{array}
$$
Hence, $|c|^2  = \tfrac{3(2n+1)}{2}  $ and the second equation of Proposition \ref{final} follows from
$$
\pr_{\Sym^2H\Sym^2E} \, \delta^* \delta =  (\pr_{\Sym^2H\Sym^2E} \circ \delta^*)\,
(\pr_{\Sym^2H\Sym^2E} \circ \delta^*)^* = |c|^2  D^*_{-1,-1} \, D_{-1,-1}  = \tfrac{3(2n+1)}{2}  B_{-1,-1} \ .
$$

Since we are again on symmetric $2$-tensors the operator $\Delta = \nabla^*\nabla + q(R)$  on sections of 
$\Sym^2H\Sym^2E$ coincides with the Lichnerowicz-Laplacian. 
We write $\Delta  = \nabla^*\nabla - q(R) + 2q(R)$ and substitute the universal formula for $\nabla^*\nabla$
and the universal Weitzenb\"ock formulas for $-q(R)$ given above. Then
$$
\begin{array}{ccl}
\Delta &=&  \frac32 B_{-1,2} + 3 B_{1,1}    - 3n B_{-1,-1}   + 2q(R) \\[1ex]
&=&
 \left(\frac32 B_{-1,2} + 3 B_{1,1} + \frac32 B_{-1,-1} \right) - \frac{3(2n+1)}{2}  B_{-1,-1} + 2q(R) \ .
\end{array}
$$
Comparing the last equation with $\Delta = \delta \delta^* - \delta^* \delta  +2q(R) $ and 
substituting the second equation of Proposition \ref{final}, which we already proved, finishes the proof of Proposition  \ref{final}.
\qed


%

\medskip

Using these formulas we will now  study the eigenspaces $(\Sym^2H\Sym^2E)_\lambda$.
We start with a characterisation of the eigenspace for the minimal eigenvalue $\lambda_1= \frac{n+1}{n+2}2 \E$.

\begin{epr}\label{divergence2b}
Let $(M^{4n}, g)$ be a positive quaternion-K\"ahler manifold. Then for any symmetric $2$-tensor
$h \in ( \Sym^2H\Sym^2E)_{\lambda}$ with $\lambda_1 \le \lambda \le \lambda_2$ the following is true:
\begin{enumerate}
\item 
If $\lambda = \lambda_1$ then $h$ is divergence free, i.e. $\delta h = 0$.
\medskip
\item 
Conversely, if $\delta h = 0$, then $\lambda = \lambda_1$.
\end{enumerate}
\end{epr}
\proof
(1)
For any $h \in (\Sym^2H\Sym^2E)_\lambda  $ we have  $\lambda \ge \lambda_1$.
In the limiting case $\lambda = \lambda_1$  the generalised gradients $D_{1,2}, D_{1,-1}$ and $D_{-1,-1}$  vanish  on $h$
(see \cite{H1}, Sec. 8.3, p.687).
Indeed, with the Weitzenb\"ock formulas on $\Sym^2H\Sym^2E$ we can rewrite $\Delta $ as
\beq\label{wbf8}
\Delta = \frac94 B_{1,2} + \frac{n+5}{2} B_{1,-1} + (2n+1) B_{-1,-1} +  \frac{n+1}{n+2}\frac{\scal}{2n} \ .
\eeq
Then  $D_{-1,-1}h =0$ implies  $\pr_{\Sym^2H\Sym^2E} \, \delta^* \delta h =0$ by Proposition \ref{final} and $\delta h = 0$ because of
\beq\label{pro}
0 = \la \pr_{\Sym^2H\Sym^2E} \, \delta^* \delta h, h\ra_{L^2}   =   \la  \delta^* \delta h, h  \ra_{L^2} = \|\delta h \|^2 \, .
\eeq
This proves statement (1) of Proposition \ref{divergence2b}.

\noindent
(2)
Conversely, if $\delta h = 0$  for any $h \in (\Sym^2H\Sym^2E)_\lambda   $  we immediately have $D_{-1,-1}h =0$ by Proposition \ref{final}. Moreover, since the minimal 
eigenvalues on $\Sym^3 H\Lambda^{2,1}_0E$ and  $\Sym^3 HE$ are $\frac{2{n+1}}{n+2} 2\E$ and $4\E$, respectively,
thus both larger than $\lambda_2 = 2\E$, we also have $D_{1,2} h = 0 = D_{1,-1}h$  if $\lambda \le \lambda_2$. Thus, 
$h \in (\Sym^2H\Sym^2E)_{\lambda_1}$ by \eqref{wbf8}. This finishes the proof of statement (2). 
\qed

%
%

\medskip

Next, we give  a description of the eigenspaces $(\Sym^2H\Sym^2E)_\lambda$ for $\lambda_1 < \lambda \le \lambda_2$.
This case is analogous to the case of $\Lambda^2_0E$ (see Proposition \ref{LambdaE}), i.e. we have the following statement.

\begin{epr}\label{symsym}
Let  $(M^{4n}, g)$ be a positive quaternion-K\"ahler manifold. Then for any eigenvalue
$\lambda$ with $\lambda_1 < \lambda \le \lambda_2$  it holds that \; $(\Sym^2H\Sym^2E)_\lambda \cong C^\infty(M)_\lambda$.
\end{epr}
\proof
With  arguments as for  $\Lambda^2_0E$ we will show that for all $\lambda $ with $\lambda_1 < \lambda <  \lambda_2$ 
the map $D_{-1,-1} : (\Sym^2H\Sym^2E)_\lambda \rightarrow (HE)_\lambda   \cong C^\infty(M)_\lambda$ 
defines an isomorphism.

First, we note that the map $D_{-1,-1}$ is injective for all $\lambda $ with $\lambda_1 < \lambda \le \lambda_2$. Indeed, if
$D_{-1,-1} h = 0$ then it follows $\delta h = 0$ by Proposition \ref{final} and the argument in \eqref{pro}.
Then Proposition \ref{divergence2b},   implies that  $\lambda = \lambda_1$ what we have excluded, thus $h=0$.

The surjectivity of $D_{-1,-1}$ is then again a consequence of the relative dimension formula and \eqref{eq4}. For 
$X \in (HE)_\lambda $ we have $D_{-1, -1} (D_{-1,-1})^* X = c_1 D_{1,1}^* D_{1,1} X = c_2 X$ for
explicit non-vanishing constants $c_1, c_2$. In the case $\lambda = \lambda_2$ we note that $\im  D_{-1,-1} $ is orthogonal
to $\iso(M,g)$ and take a vector field $X = df$, using arguments as for $\Lambda^2_0E$. 
\qed

\medskip

\begin{ere}
As in the case of  $\Lambda^2_0E$,   the identification of $(\Sym^2H\Sym^2E)_\lambda \subset \Gamma(\Sym^2 \T M)$ with $d C^\infty(M)_\lambda$,
for $\lambda_1 < \lambda \le \lambda_2$, can also be defined by the divergence $\delta$. The injectivity is a consequence of Proposition \ref{divergence2b}. The
surjectivity follows from the proof of Proposition \ref{symsym} and Proposition \ref{divergence}.
\end{ere}

\medskip

%
\section{Infinitesimal Einstein deformations and destabilising directions} \label{final1}
%

\begin{ath}\label{main2}
Let $(M^{4n}, g)$ be a positive quaternion-K\"ahler manifold.
Then the space of destabilising directions of the Einstein metric $g$ is isomorphic to
$$
(\Sym^2H \Sym^2E)_{\lambda_1} \;   \oplus   \,  \bigoplus_{\lambda_1 < \lambda < \lambda_2}  C^\infty(M)_{\lambda} \ .
$$
\end{ath}
\proof
By definition  a destabilising direction is a section in
$(\Sym^2_0 \T M)_\lambda \cap \ker \delta$ for  a $\lambda $ with $ \lambda < \lambda_2 = 2\E$. From our eigenvalue 
estimates on symmetric $2$-tensors we have $\lambda \ge \lambda_1$. Any 
symmetric tensor  $h \in (\Sym^2_0 \T M)_\lambda$ can be written as $h=h_1+h_2$ with $h_1\in (\Lambda^2_0E)_\lambda$
and $h_2 \in (\Sym^2H\Sym^2E)_\lambda$. 

If $\lambda = \lambda_1$ then $\delta h_2 =0$ by Proposition \ref{divergence2b}. Thus, also $\delta h_1 = 0$ and it follows 
$h_1=0$ because of Corollary \ref{h2}.  Hence, $(\Sym^2_0 \T M)_{\lambda_1} \cap \ker \delta = (\Sym^2H\Sym^2E)_{\lambda_1}$.

For all eigenvalues $\lambda$ with $\lambda_1 < \lambda \le \lambda_2$ we have 
$
(\Lambda^2_0E)_\lambda \cong (\Sym^2H \Sym^2E)_{\lambda} \cong C^\infty(M)_\lambda
$
by Proposition \ref{LambdaE} and \ref{symsym}. Hence  the isomorphism $(\Sym^2_0 \T M)_\lambda \cap \ker \delta  \cong C^\infty(M)_\lambda$
 follows in this situation  from the short exact sequence
\beq\label{exact}
0 \rightarrow (\Sym^2_0 \T M)_\lambda \cap \ker \delta  \stackrel{i}{\rightarrow} (\Lambda^2_0E)_\lambda \oplus (\Sym^2H \Sym^2E)_{\lambda}
 \stackrel{\delta}{\rightarrow}d C^\infty(M)_\lambda \rightarrow 0
\eeq
where $i(h) := h_1 \oplus  h_2$ and $ \delta (h_1 \oplus  h_2) := \delta h_1 + \delta h_2$. Recall that here $d C^\infty(M)_\lambda \cong C^\infty(M)_\lambda$.
\qed

\medskip

\begin{ere}\label{wolf}
The minimal eigenspace $(\Sym^2H\Sym^2E)_{\lambda_1}$ plays an important role. 
Using the method developed in   \cite{SW2002} (see  Eq. (20) and Thm. 4.4, more details can be found in \cite{MFO}),  it can be shown that the index $i^{1,n+1}$  of the Dirac operator twisted with $\Sym^{n+1}H E$ is given as
$$
i^{1,n+1} \;=\;  - \dim (\Sym^2H\Sym^2E)_{\lambda_1} \; + \; \dim (HE)_{\lambda_1} \ .
$$

But  by Proposition \ref{LeB}, $ \dim  (HE)_{\lambda_1} \neq 0$ if and  only if  $M = \H P^n$. Hence, for  $M \neq \H P^n$
we have $ \dim (\Sym^2H\Sym^2E)_{\lambda_1} = - i^{1,n+1} $. If $M = \H P^n$ then $i^{1,n+1} 
= n(2n+3)$ by \cite{herrera}, p. 213. Hence, $ i^{1,n+1} = \dim  (HE)_{\lambda_1}  $ and we have $(\Sym^2H\Sym^2E)_{\lambda_1} = 0$ in this case. The dimension of the minimal eigenspace $ (HE)_{\lambda_1} $ can be found in 
\cite{LBL1}, Prop. 2.1.1.  or in \cite{CW}, Thm. 5.2. It turns out that  $ (HE)_{\lambda_1} $ is isomorphic to the $\Sp(n+1)$-representation
$\Lambda^2_0 \H^{n+1}$.

It is conjectured, that
the index $i^{1,n+1} $  vanishes for all positive quaternion-K\"ahler manifolds different from $\H P^n$. In \cite{herrera}
the conjecture was shown to be true for all Wolf spaces. In particular, we have  $(\Sym^2H\Sym^2E)_{\lambda_1} = 0$
for all Wolf spaces.
\end{ere}

\medskip

\begin{ath}\label{main1}
Let $(M^{4n}, g)$ be a positive quaternion-K\"ahler manifold. Then the space of infinitesimal Einstein deformations of $g$
is isomorphic to $C^\infty(M)_{2\E}$
\end{ath}
\proof
The space of infinitesimal Einstein deformations of $g$ is given by $(\Sym^2_0 \T M)_{2\E} \cap \ker \delta$.
Hence, the isomorphism of the theorem follows as above by the exact sequence \eqref{exact}.
\qed

\begin{ere}
For any $f\in C^\infty(M)_\lambda$ with $\lambda_1 < \lambda \le \lambda_2$ we can  explicitly define a divergence 
free symmetric $2$-tensor in $(\Sym^2_0 \T M)_\lambda$.  	We define $h:= h_1 + h_2$ by
$$
h_1 := -\frac23  \left( \lambda \, -\,  \frac{n+1}{2n(n+2)} \scal  \right)^{-1} \pr_{\Sym^2H\Sym^2E} \,  \delta^* d \,f
$$
and 
$$
h_2 := \frac{2n}{n-1}  \left( \lambda \, + \,  \frac{\scal }{2(n+2)}  \right)^{-1} \pr_{\Lambda^2_0E} \,  \delta^* d \,f
$$

\noindent
Since $f\in C^\infty(M)_\lambda$ it follows that  $h_1 + h_2  \in (\Sym^2_0 \T M)_\lambda$. Moreover, by
Proposition \ref{divergence} together with \eqref{eq3} and \eqref{eq4} we have  that $\delta h = \delta h_1 + \delta h_2 = - df + df = 0$.
Note that $\delta^* d \,f$ is up to a factor the Hessian of $f$. In particular, we obtain a parametrisation of the space
of infinitesimal Einstein deformations through a suitable linear combination of projections of the Hessian of $f$.
This is similar to the parametrisation in the case of the complex Grassmannian  $\mathrm{Gr}_2(\C^{n+2})$  given in \cite{NS}.
\end{ere}


The two theorems can be reformulated,  linking stability, improved eigenvalue estimates
and the vanishing of the index $i^{1,n+1}$.

\begin{ecor}
Let $(M^{4n}, g)$ be a positive quaternion-K\"ahler manifold, not isomorphic to $\H P^n$. Consider the 
following statements:
\begin{enumerate}
\item 
The first non-vanishing eigenvalue $\lambda$ on $C^\infty(M)$ satisfies $\lambda \ge 2 \E$ (resp. $\lambda >2\E$)
\medskip
\item
The first  eigenvalue $\lambda$ on $\Gamma(\Sym^2_0 \T M)$ satisfies $\lambda \ge 2 \E$ (resp. $\lambda >2\E$)
\medskip
\item
The metric $g$ is semistable (resp. strictly stable).
\end{enumerate}
Then (2) and (3) are equivalent and (3) implies (1). If in addition $i^{1,n+1} = 0$ then all three statements
are equivalent. Moreover, if $g$ is stable then $i^{1,n+1} = 0$.
\end{ecor}

\medskip

A consequence of Theorem \ref{main2} and the definition of $\nu$-stability is the following corollary.

\begin{ecor}
Let $(M^{4n}, g)$ be a positive quaternion-K\"ahler manifold  different from  $\H P^n$  and assume that $i^{1,n+1} =0$ or equivalently that
$(\Sym^2H\Sym^2E)_{\lambda_1} = 0$. Then the  metric $g$ is stable if and only it is $\nu$-stable.
\end{ecor}

\medskip

In the final part of this section we consider our results in the case of Wolf spaces, the up to now only
known examples of positive quaternion-K\"ahler manifolds.

As already mentioned in the  introduction, it follows from the work of Koiso (see \cite{Koiso} and \cite{SW2022})  
that all Wolf spaces different from $\Gr_2 (\C^{n+2})$  are stable, i.e. there are no destabilising directions. The complex Grassmannian $\Gr_2 (\C^{n+2})$ is semistable and admits infinitesimal Einstein deformations.

 Hence, from Theorem \ref{main2} and Remark \ref{wolf},
we obtain a new proof for the vanishing of the indices  $i^{1,n+1} $ on the Wolf spaces.

Conversely, Theorem \ref{main2} can be used to check the stability of the Wolf space. Indeed, we have
 $(\Sym^2H\Sym^2E)_{\lambda_1} = 0$ for all Wolf spaces as mentioned in Remark \ref{wolf}. Moreover, it follows
 from  \cite{Mil} that for all Wolf spaces different from $\H P^n$ the minimal eigenvalue $\lambda$ of the Laplacian
 on functions is larger than $\lambda_2 = 2 \E$. Hence, for all Wolf spaces there is no minimal eigenvalue between
 $\lambda_1$ and $\lambda_2$. Theorem  \ref{main2}  now implies that no Wolf space admits destabilising
 directions. 

\medskip

The only Wolf space with infinitesimal Einstein deformations is the complex Grassmannian $\mathrm{Gr}_2(\C^{n+2})$ \cite{Koiso}.
But  $\mathrm{Gr}_2(\C^{n+2})$ with its symmetric metric is K\"ahler-Einstein of positive scalar curvature. Hence, $2\E$ is the smallest
Laplace eigenvalue on non-constant functions and, by the Matsushima theorem,   $C^\infty(M)_{2\E}$ is isomorphic to the space
of Killing vector fields $\iso(M, g)$. There is also an explicit isomorphism between the space of   infinitesimal Einstein deformations and
the space of Killing vector fields \cite{NS}.



\appendix \label{A1}
\section{}
Besides the index $i^{1,n+1}$, also the index $i^{2,n}$ of the Dirac operator twisted with $\Sym^nH \Lambda^2_0E$ plays a special role.
First it can be computed as
\beq\label{index1}
i^{2,n} = \dim (\Sym^2H \Lambda^{2,2}_0E)_{\lambda_3} - \dim (H\Lambda^{2,1}_0E)_{\lambda_3} \ .
\eeq
where $\lambda_3 := \frac{n}{n+2} 2\E = \frac{\scal}{2(n+2)}$ is the minimal eigenvalue of $\Sym^2H \Lambda^{2,2}_0E$
and also of $H\Lambda^{2,1}_0E$.
Here we use again the formulas of
\cite{SW2002} (see also \cite{MFO} for more details).  Note that the summand $\dim (\Lambda^2_0E)_{\lambda_3}$
in the index formula of \cite{MFO}  vanishes because of the improved eigenvalue estimate on $\Lambda^2_0E$ (see \cite{H1}).
Moreover it can be shown that $(H\Lambda^{2,1}_0E)_{\lambda_3} \cong (\Sym^2 E)_{\lambda_3}$.
The isomorphism is given by the generalised gradient $D_{-1,1}$, as can be seen by  the vanishing of $5$ out
of $10$ generalised gradients in the limit case of the eigenvalue estimate on $H\Lambda^{2,1}_0E$ and 
Proposition \ref{injective1}.completeness

The index $i^{2,n}$  in the case of the Wolf space $M= \mathrm G_2/\SO(4)$, i.e. $n=2$, can be computed using a formula
by A. Fino and S. Salamon (\cite{SS}, Prop. 8.4). In this special it follows that
$$
i^{0,4} - i^{1,3} + i^{2,2} = 2 \chi(M) - b_2(M)  + b_4(M) \ .
$$ 
Here $i^{0,4}$ is the dimension of the isometry group, i.e. $i^{0,4}=14$. We already remarked that $ i^{1,3}=0$
for all Wolf spaces different from $\H P^n$. For all positive quaternion K\"ahler manifolds the odd Betti numbers 
vanish and $b_2(M)=0$, if $M$ is different from $\Gr_2 (\C^{m+2})$. In the case $M= \mathrm G_2/\SO(4)$ we also have
 $b_4(M)=1$. Hence the Euler characteristic is $3$ and using the above formula we find
 $
 i^{2,2} = -7
 $.
 Note that there is a sign error in the computation of $ i^{2,2}$ given in \cite{herrera}. It follows from \eqref{index1}
 that $(H\Lambda^{2,1}_0E)_{\lambda_3} $ and thus also  $(\Sym^2 E)_{\lambda_3}$ is non-trivial on $M= \mathrm G_2/\SO(4)$. Hence, 
 the minimal eigenvalue $\lambda_3$ is attained on $\Sym^2 E$ and we see once again that  our improved 
eigenvalue estimate of Proposition \ref{improved} is optimal.


%
\section*{Acknowledgments}
 The first author was partially supported by JSPS KAKENHI Grant Number JP24K06721.
 The second author acknowledges the support received by the Special Priority
 Program SPP 2026 {\em Geometry at Infinity} funded by the Deutsche
 Forschungsgemeinschaft DFG. 
 
\bigskip


\end{document}